\documentclass[reqno,english,11pt]{amsart}
\usepackage{amsmath,amsfonts,amssymb,graphicx,amsthm,url}
\usepackage[noadjust]{cite}
\usepackage{stmaryrd}
\usepackage{mathrsfs,booktabs,tabularx}
\usepackage{xifthen,xcolor,tikz,setspace}
\usetikzlibrary{decorations.pathmorphing,patterns,shapes,calc,decorations}
\usetikzlibrary{decorations.pathreplacing}
\usepackage{mathtools,upgreek,paralist}
\usepackage[final]{showkeys} 

\definecolor{refkey}{gray}{.75}
\definecolor{labelkey}{gray}{.5}
\usepackage[algo2e,boxed,vlined,algoruled]{algorithm2e}
\usepackage{comment}
\usepackage[shortlabels]{enumitem}

\usepackage[letterpaper]{geometry}
\usepackage[colorinlistoftodos]{todonotes}
\presetkeys{todonotes}{inline, color=green}{}

\usepackage{accents}
\usepackage[algo2e,boxed,vlined,algoruled]{algorithm2e}

\usepackage[small]{caption}
\usepackage[colorlinks=true]{hyperref}
\colorlet{DarkGreen}{green!50!black}
\colorlet{DarkGray}{gray!60!black}

\numberwithin{equation}{section}

\renewcommand{\epsilon}{\varepsilon}

\usepackage{color}
 \definecolor{refkey}{gray}{.5}
 \definecolor{labelkey}{gray}{.5}
\definecolor{light}{gray}{.9}

\usepackage{soul}
\usepackage{float} 

\numberwithin{equation}{section}

\newtheorem{theorem}{Theorem}[section]
\newtheorem{corollary}[theorem]{Corollary}
\newtheorem{lemma}[theorem]{Lemma}
\newtheorem{proposition}[theorem]{Proposition}

\theoremstyle{definition}
\newtheorem{definition}[theorem]{Definition}
\newtheorem{remark}[theorem]{Remark}

\DeclareMathOperator{\diameter}{diam}
\DeclareMathOperator{\card}{Card}
\DeclareMathOperator{\LE}{LE}
\DeclareMathOperator{\SLE}{SLE}
\DeclareMathOperator{\lv}{lv}
\DeclareMathOperator{\disc}{Dis}

\newcommand{\C}{\mathbb{C}}

\newcommand{\Z}{\mathbb{Z}}
\newcommand{\N}{\mathbb{N}}

\newcommand{\dis}{^{\#\delta}}

\title{Universality for the 2D Random Walk Loop Soup\\
	{}}
\author{Yihao Pang}
\address{Université Paris Cité and Sorbonne Université, CNRS, Laboratoire de Probabilités, Statistique et Modélisation, F-75013 Paris, France}
\email{ypang@lpsm.paris}

\date{\today}
\begin{document}
    \begin{abstract}
	   We show that the scaling limit of the random walk loop soup on suitable planar graphs is the Brownian loop soup, under a topology on multisets of unrooted, unparameterized, and macroscopic loops. The result holds assuming only convergence of simple random walk to Brownian motion, a Russo-Seymour-Welsh type crossing estimate, and the bounded density of the graphs. The proof relies on Wilson's algorithm and Schramm's finiteness theorem. Precisely, we approximate the random walk loop soup by the set of loops erased in a greedy variant of Wilson's algorithm, thereby establishing convergence. The resulting limit is identified using the result of Lawler and Trujillo Ferreras \cite{lawler2004randomwalkloopsoup}. 
    \end{abstract}
    \maketitle

\section{Introduction}
The Brownian loop soup (BLS), introduced by Lawler and Werner~\cite{lawler2003brownianloopsoup}, is a fundamental object in random conformal geometry. A BLS in a domain $D$ is a random collection of Brownian-type loops in $D$, which can be viewed as a Poisson point process in the space of unrooted loops, with intensity called the Brownian loop measure. 
It satisfies both conformal invariance and a restriction property and is closely connected to other conformally invariant random processes such as Schramm-Loewner evolution (SLE), the Conformal loop ensemble (CLE), and the Gaussian free field. In particular, the cluster of a Brownian loop soup was used to construct the CLE~\cite{werner2003slesboundariesclustersbrownian,sheffield2011conformalloopensemblesmarkovian}.

A discrete analogue of the Brownian loop soup was first introduced by Lawler and Trujillo Ferreras~\cite{lawler2004randomwalkloopsoup}. They defined the random walk loop soup (RWLS) on the square lattice and proved its convergence to BLS. The proof relies on a strong KMT approximation, which enables a coupling between macroscopic random walk loops and macroscopic Brownian loops.
Qian~\cite{qian2026couplingbrownianloopsoups} strengthened both this coupling and its extension to higher dimensions~\cite{MR3877544} to all loops at polynomial scales, but still within the lattice setting.
Le Jan~\cite{jan2008markovloopsrenormalization} introduced the Markovian loop soup, a continuous-time generalization of RWLS associated with a Markov chain. At least in the discrete setting, the occupation field of the Markovian loop soup is equal to half of the square of the discrete Gaussian free field~\cite{jan2008markovloopsrenormalization,jan2010markovpathsloopsfields}. 

The principle of universality suggests that these results should hold for a much broader class of planar graphs.
In this paper, we consider a more general underlying graph (than $\Z^2$) which satisfies three assumptions:
(i) Invariance Principle, (ii) the bounded density of the graphs and (iii) a Russo-Seymour-Welsh type crossing estimation. These assumptions are introduced and discussed by Berestycki, Laslier and Ray~\cite{berestycki2018dimersimaginarygeometry,berestycki2024dimersriemannsurfacesii,berestycki2024dimersriemannsurfacesi}, see Section \ref{graph} for details. 
Our main result in the present paper is as follows:
\begin{theorem}\label{scaling limit}
	Let $D$ be a simply connected and bounded domain. Let $\{\Gamma ^{\#\delta}\}_{\delta>0}$ be a sequence of graphs, all of which satisfy the assumptions of Section \ref{graph}. Let $\{\mathcal{L}_D^{\#\delta}\}_{\delta>0} $ be the RWLS on $\{\Gamma^{\#\delta}\}_{\delta>0}$ restricted to $D$. Let $\mathcal{L}_D$ be the BLS restricted to $D$. Then in the sense of the loop soup topology defined in Section \ref{topology}, as $\delta \to 0$, the random walk loop soup converges in law to Brownian loop soup. 
	\[ \mathcal{L}_D^{\# \delta} \xrightarrow[\delta \to 0]{(d)} \mathcal{L}_D. \]
\end{theorem}

The proof is based on two tools. First, we use the relationship between the RWLS and the loops erased in Wilson's algorithm. Wilson's algorithm~\cite{10.1145/237814.237880} is a method to generate a spanning tree of a given graph by iteratively performing loop-erased random walks (LERW). 
It was shown in \cite{Lawler_Limic_2010} that the loops erased in Wilson's algorithm can be obtained by concatenating the loops in RWLS;  see Proposition~\ref{loop addition} and Corollary~\ref{loop addtion2} for more details.
We define a ``greedy'' Wilson's algorithm, first introduced by Berestycki, Laslier, and Ray~\cite{berestycki2024dimersriemannsurfacesii}, which separates each random walk path in Wilson's algorithm into a concatenation of tiny paths and approximates some tiny paths by their loop-erasure. This ensures that we can ``ignore'' most microscopic loops and approximate the macroscopic loops erased in the greedy Wilson's algorithm. To measure the distance between the loops erased during the sampling of finite branches in this algorithm and the macroscopic loops in RWLS that attach to these branches, we use the following observation: as the mesh size decreases, the macroscopic loops erased in this algorithm have distinct roots and hence approximate the macroscopic loops in RWLS; see Section \ref{Concatenation probability}. 
The loops erased during the sampling of the remaining branches remain microscopic by Schramm's finiteness theorem~\cite{schramm1999scalinglimitslooperasedrandom,berestycki2024dimersriemannsurfacesii}; see also Section \ref{section schramm's finiteness lemma}.

In particular, in this paper, the invariance principle assumption is assumed for the convergence of segments in the greedy Wilson's algorithm, while the remaining assumptions are for Schramm's finiteness theorem.

An immediate corollary is that for any geometrical property of unparameterized loops, we can derive the convergence of the loop measure of the loops satisfying these properties. Here is a list of workable properties:
\begin{itemize}
	\item The loops that have diameter at least $d$ with $d>0$.
	\item The loops that touch a finite connected closed set $B_1, \cdots , B_n$ of positive size.
	\item The loops that are non-contractible and stay in the complement of finitely many connected open sets $O_1, \cdots, O_n$.
\end{itemize}

\begin{corollary}\label{cor}
	Under the same hypothesis as Theorem~\ref{scaling limit}, let $\mu^{\#\delta}_D$ (resp. $\mu_D$) be the random walk loop measure (resp. the Brownian loop measure) restricted to $D$. Let $A$ be an open set in the space of unparameterised loops such that $A\subset \{l: \diameter (l) \ge \epsilon\}$ for some $\epsilon >0$ and $\mu_D(\partial A)=0$. Then
	\[ \mu^{\#\delta}_D (A)\xrightarrow[\delta \to 0]{} \mu_D (A). \]
\end{corollary}

The paper is structured as follows: In Section \ref{background}, we introduce the graph setting and review key facts on loop soup, LERW and Wilson's Algorithm. We also define the metric space in which the loop soup lies.
In Section \ref{chapter GWA}, we introduce both the discrete and continuum versions of the Greedy Wilson's Algorithm and establish the scaling limit result for the loops erased in these algorithms. In Section \ref{chapter proof}, we derive Theorem \ref{scaling limit} following the strategy outlined above, and we prove Corollary \ref{cor}.

\section{Background and setup}\label{background}
\subsection{Graph discretization and assumptions} \label{graph}
In this section, we specify the class of planar graphs considered in this work. This setup is adapted from the conditions introduced and discussed by Berestycki, Laslier and Ray~\cite{berestycki2018dimersimaginarygeometry,berestycki2024dimersriemannsurfacesii,berestycki2024dimersriemannsurfacesi}.
Specifically, we consider a sequence of graphs $\{ \Gamma^{\#\delta} \}_{\delta >0}$ , where each $\Gamma^{\#\delta} =(V(\Gamma^{\#\delta}), E(\Gamma^{\#\delta})) $ is an oriented, infinite, and connected planar graph properly embedded in the plane. We associate a weight  $w(u,v)$ for every oriented edge $(u,v)\in E(\Gamma^{\#\delta})$. We consider the continuous time random walk where the walker jumps from $u$ to $v$ with rate $w(u,v)$. We call it the \textbf{Random Walk} on $\Gamma^{\#\delta}$. Precisely, for each vertex $v$, the holding time on $v$ is an exponential distribution with rate $\sum_{(u,v)\in E(\Gamma^{\#\delta})} w (u,v)$, then the walker will jump to its neighbor $v$ with transition probability $q(u,v)=\frac{w(u,v)}{\sum_{(u,v')\in E(\Gamma^{\#\delta})} w(u,v')} $. Let us emphasize that we allow nonreversible chains. 
We assume that $\{ \Gamma ^{\# \delta}\}_{\delta}$ satisfy the following properties:

\begin{enumerate}[label=(\roman*)]
	\item \textbf{(Invariance principle)} As $\delta \to 0$, the continuous time random walk $\{  X^{\#\delta}_t \} _{t\ge 0}$	on $\Gamma ^{\# \delta}$ started from the nearest vertex to $0$ converges in law to a standard Brownian motion $(B_t, t\ge 0)$ in the following sense: For every compact set $K\subset \C$ containing $0$ in its interior, let $\tau_K$ denote the exit time of $K$,  in the sense of uniform topology on curves up to parametrization, we assume
	\[ (X^{\#\delta}_{t})_{0\le  t\le \tau_K} \xrightarrow[\delta \to 0]{(d)} (B_{t})_{0\le t\le \tau_K}. \] 
	\item \textbf{(Bounded density)} We assume that there exists a constant $C$ independent of $\delta$, such that the number of vertices of $\Gamma^{\# \delta}$ in any ball of radius $\delta$ is smaller than $C$.
	\item \textbf{(Uniform crossing estimate)} Let $R$ be the horizontal rectangle $[0,3] \times [0,1]$ and $R'$ be the vertical rectangle $[0,1]\times [0,3]$. Let $B_1 := B((1/2,1/2),1/4) $ be the starting ball and $B_2:= B((5/2,1/2),1/4)$ be the target ball. There exist universal constants $\delta_0 , \alpha _0> 0$, such that for all $\delta < \delta_0$, the following is true. For all $z\in \mathbb{C}$, $l\ge 1/\delta_0$, for all $v\in \Gamma^{\#\delta} \cap ( l \delta B_1 +z )$, we have
	\[ \mathbb{P}_{v} (X^{\#\delta} \textit{ hits  } (l\delta B_2+z) \textit{ before exiting } (l\delta R+z))> \alpha_0, \]
	where $X^{\#\delta} $ is a random walk starting from $v$. The symmetric statement with $R'$ also holds.
	
\end{enumerate}

Let $D$ be a simply connected and bounded domain. We choose the crossing point of the edges of $\Gamma^{\#\delta}$ and the boundary of the domain $\partial D$ as the boundary of the graph, denoted by $\partial \Gamma^{\#\delta}_D$. 
We will only consider the random walk in $D$ until it hits the boundary, so we assume that the random walk is killed at that time. 
Also, with a slight abuse of notation, we will generally regard the trajectory of a random walk as a continuous path via linearization along edges used in order. 
\begin{remark}
	Throughout this paper, we adopt the continuous-time random walk framework. However, since all assumptions and results depend only on the underlying topology up to reparameterization, the use of a discrete-time random walk would yield no essential difference. 
\end{remark}

Two immediate results of the above assumptions are the following lemma:
\begin{lemma}[Beurling type estimate]\label{beurling estimate}
	For all $r,\epsilon>0$, there exist $\eta>0$ and $\delta(\eta)$ depending on $\eta$ such that for any $\delta<\delta(\eta)$, for any vertex $v\in \Gamma^{\#\delta}$ such that $d(v, \partial D)<\eta$, the probability that a simple random walk emanating from $v$ exits $B(v,r)$ before hitting $\partial \Gamma ^{\#\delta}$ is at most $\epsilon$, i.e.
	\[ \mathbb{P}_v (X^{\#\delta} \textit{ exits } B(v,r) \textit{ without hitting } \partial \Gamma^{\#\delta}_D ) <\epsilon. \]
\end{lemma}
\begin{proof}
	This is a restatement of Lemma 2.6 from \cite{berestycki2024dimersriemannsurfacesii}. As the proof carries over directly to the present context, we do not reproduce it here.
\end{proof}
\begin{lemma}[The size of edges is uniformly small]\label{uniform small edge}
	Assuming only the uniform crossing estimate,  as $\delta \to 0$, we have \[ \sup_{e\in \Gamma \dis} \diameter(e) =o(1). \]
\end{lemma}

\begin{proof}
	We prove this result by contradiction. Assume that there exists $\epsilon>0$, such that  for all $n\in \mathbb{N}$, there exists $\delta_n <1/n$ and an edge $e^{\#\delta_n} \in \Gamma^{\#\delta_n}$, such that $\diameter (e^{\#\delta_n})>\epsilon$. Let $e'^{\#\delta_n}$ be a continuous portion of $e^{\#\delta_n}$ with diameter $\epsilon$. Without loss of generality, assume that the line segment connecting the two points realizing the diameter of $e'^{\#\delta_n}$ makes an angle of at most $\pi /4$ with the vertical axis. By the uniform crossing estimate, we use the same notation as in the assumption, take $\delta_n$ such that $\delta_n /\delta_0 < \epsilon /4$, then there exists $z\in \C$ such that the two balls $\frac{\sqrt{2}\epsilon}{2} B_1 +z$ and $\frac{\sqrt{2}\epsilon}{2}B_2 +z$  in the horizontal rectangle $\frac{\sqrt{2}\epsilon}{2} R + z$ are separated by the edge $e^{\#\delta_n}$. In this case, the probability of the simple random walk starting from an abstract vertex in the left ball $\frac{\sqrt{2}\epsilon}{2} B_1 +z$  hitting the right ball $\frac{\sqrt{2}\epsilon}{2} B_2+z$ before exiting the rectangle $\frac{\sqrt{2}\epsilon}{2} R +z$ is $0$, which contradicts the uniform crossing estimate assumption.
\end{proof}

\subsection{Loop measures and loop soup} \label{topology} In this section, we define the loop soup and the topological space in which it lies.
\paragraph*{\textbf{Space of curves}}
Let $\mathcal{C}$ be the space of all continuous curves from $[0,1]$ to $\C$ modulo reparameterizations. 
Define the metric on $\mathcal{C}$ as follows: for any two curves $\gamma_1, \gamma_2$, \[ d_{curve} (\gamma_1, \gamma_2) =\inf \sup_{0\le t\le 1 } |\gamma'_1(t)-\gamma'_2(t)|, \] 
where the infimum is taken over all parameterizations $\gamma_1',\gamma_2'$ of $\gamma_1,\gamma_2$. The metric space $(\mathcal{C},d_{curve})$ is separable and complete, see \cite{kemppainen2015randomcurvesscalinglimits}. We note that the parameterization from $[0,1]$ is used to ensure the positivity of the metric. In later contexts, a curve may be defined on a different interval, but it can always be viewed as parameterized by $[0,1]$.

\paragraph{\textbf{Space of loops}}
Let $\tilde{\mathcal{C}}$ be the set of all rooted loops in $\mathcal{C}$, i.e. curves whose starting and ending points coincide. Define an equivalence relation $\sim $ on $\tilde{\mathcal{C}}$ by setting $\gamma_1 \sim \gamma _2$ if and only if there exists a right translation $\theta_r$ of $\mathbb{R}$ for some $r\ge 0$ and two corresponding parameterized representations $\gamma_1'$ and $\gamma_2'$, such that $\gamma'_1 =\gamma'_2 \circ \theta_r $ (here we extend the rooted loop to $\mathbb{R}$ by its periodicity). We define the space of (unrooted) loops as the quotient space $\mathscr{L}:=\tilde{\mathcal{C}} /\! \sim$. We define the metric on $\mathscr{L}$ as follows: for $\gamma_1$, $\gamma_2 \in \mathscr{L}$, 
\[ d(\gamma_1, \gamma_2)=\inf _{0\le r\le 1;\gamma_1',\gamma_2'} d_{curve}(\gamma'_1, [\theta_r\circ \gamma'_2]), \] where $\gamma_2'$ is a parameterization of $\gamma_2$ and $[\theta_r\circ \gamma'_2]$  is the equivalence class of $\gamma_2'$ after translation. It is easy to see that $d$ is a metric. The separability of $(\mathscr{L},d)$ follows directly from the separability of $(\mathcal{C},d_{curve})$. For the completeness, it suffices to notice that for each Cauchy sequence in $(\mathscr{L},d)$, we can choose simultaneously parameterizations such that the parameterized sequence is Cauchy in $(\mathcal{C}, d_{curve})$. For a domain $D\subset \C$, we denote the loop space restricted to $D$ by $\mathscr{L}_D$.

\paragraph*{\textbf{Brownian loop soup}}
The \textit{Brownian loop measure} is the measure $\mu$ on $\mathscr{L}$ induced by the measure on $\mathcal{C}$
\[ \int _{\mathbb{C}} \int_0^{+\infty} \frac{1}{2\pi t^2} \mu^{br}_t(z) \, dt\, dz, \]
where $\mu^{br}_t(z)$ is the law of a Brownian bridge from $z$ to $z$ with life time $t$. The \textit{Brownian loop soup} over $\mathscr{L}$ is a Poisson point process with intensity $\mu$, and we denote it by $\mathcal{L}$. For a domain $D\subset \mathbb{C}$, we denote the Brownian loop soup restricted to $D$ by $\mathcal{L}_D$.
\begin{remark}
	It is also common to define the loop soup as a Poisson point process with intensity $Leb_{\mathbb{R_+}}\otimes \mu$ and denote it by $(\mathcal{L}_{\alpha},\alpha \ge 0)$. One of the advantages of this definition is that we have a natural coupling for all time $\alpha \in \mathbb{R}_+$.  In this paper, we are only concerned with the multiset of loops $\mathcal{L}_1$, which we denote by $\mathcal{L}$. 
\end{remark}

\paragraph*{\textbf{Random Walk Loop Soup}}
A rooted loop is a finite path $l=(v_0,v_1,\cdots,v_p)$ with $v_0=v_p$. We denote the length of $l$ by $|l|$. The \textit{rooted loop measure} of $\Gamma$ is a measure $\mu^{rooted}$ defined on the set of all rooted loops by
\[\mu^{rooted}(l)=\frac{\Pi _{i=0}^{|l|-1}q(v_i,v_{i+1})}{|l|}. \]
An unrooted loop is an equivalence class of rooted loops modulo translations.
This induces the \textit{random walk loop measure} of $\Gamma$, defined on the set of all unrooted loops by
$$\mu^{\Gamma} ([l])=\sum_{l'\in[l]} \mu^{rooted} (l'). $$
To ensure that the discrete loops lie in the loop space $\mathscr{L}$, we view a discrete loop as a continuous function via linearization along edges. By Lemma \ref{uniform small edge}, this does not attack the proof of Theorem \ref{scaling limit}.  The \textit{random walk loop soup} is a Poisson point process on $\mathscr{L}$ with intensity $\mu^{\Gamma}$. 
We denote the unrooted loop soup over $\Gamma^{\#\delta}$ defined in section \ref{graph} by $\mathcal{L}^{\#\delta}$.

\paragraph*{\textbf{Markovian Loop Soup}} A based loop is an element $l=(v_1,\tau_1, \cdots ,v_p, \tau_p, v_{p+1}, \tau _{p+1})$ in $\bigcup_{p\in \N} (\Gamma \times (0,+\infty))^{p+1}$ such that $v_{p+1}=v_1$ and $v_{i+1} \not= v_i$ for $i=1, \cdots,p$. We denote the number of vertices in $l$ by $p(l)$. The measure on the based loops is defined as 
\[ \mu^b =\sum_{v\in \Gamma} \int_0^{\infty} \frac{1}{t} \mathbb{P}_{t,v}^v dt, \]
where $\mathbb{P}_{t,v}^v$ is the non-normalized bridge measure from $v$ to $v$ with duration time $t$, precisely, $ \mathbb{P}^v_{t,v} (\cdot) =\mathbb{P}^v_t (\cdot \cap \{ X_t=v\} ) $. We can define an operation which maps a based loop to a rooted loop by forgetting the time duration at each vertex, precisely, $\disc (l)=(v_1, v_2,\cdots , v_{p+1})$. Moreover by \cite{jan2010markovpathsloopsfields,chang2014markovloopsdiscretespaces}, 
for an rooted loop $(w_1,\cdots, w_{p+1}) $, we have
\[ \mu^b (\{l:\disc (l)=(w_1,w_2,\cdots, w_{p+1})\} )= \frac{\Pi_{i=1}^p  q(w_i,w_{i+1}) }{p} =\mu^{rooted} ((w_1, \cdots, w_{p+1})).  \]
With this formula, we can view the random walk loop measure as an induced measure from $\mu^b$ by forgetting the time duration and the root. 

\begin{remark}
	In this paper, we work with the continuous-time random walk, which naturally leads us to introduce the Markovian loop soup. However, since we are not concerned with the time parameterization of loops, the main result in this paper holds equally for both the random walk loop soup and the Markovian loop soup, as they coincide in the loop soup space defined below.
\end{remark}

\paragraph{\textbf{Space of Loop Soups}}
In this paper, we regard the loop soup as an at most countable multiset of loops, where the number of macroscopic loops (loops with diameter larger than some positive number) is almost surely finite. We define
\[ M(D) =\{ (l_n)_{n\in \N} \in \mathscr{L}^{\N} _D: \forall \epsilon >0, \card \{ n\in \N: \diameter l_n \ge \epsilon \}<+\infty  \}, \]
where $\diameter$ means the diameter of a closed set, $\card$ means the cardinal of a set. Define a nonnegative function $d_{M(D)}: M(D) \times M(D) \to \mathbb{R}_+$ as follows, for $(X_i)_{i\in \N }\in M(D)$ and $(Y_j)_{j\in \N} \in M(D)$,
\begin{align*}
	d_{M(D)} & ((X_i)_{i\in \N}, (Y_j)_{j\in \N}) \\
	&:=\inf_{(\lambda :I\to \N)}  \sup \{ d (X_i,Y_{ \lambda (i)}), \forall i\in I; \frac{\diameter(X_i)}{2}, \forall i\in \N \backslash I; \frac{\diameter (Y_j)}{2}, \forall j\in \N \backslash \lambda (I)\} ,
\end{align*}
where the infimum is taken over all finite sets $I\subset \N$ and all injective functions $\lambda : I \to \N$.

\begin{lemma}\label{space of loop soups}
	\begin{enumerate}[label=(\roman*)]
		\item The function $d_{M(D)}$ is a pseudometric.  
		\item Let $\sim$ be the equivalence relation induced by vanishing the pseudometric $d_{M(D)}$. Define $\mathcal{M}(D):=M(D)/\!\sim$. Let $d_{\mathcal{M}(D)}$ be the induced metric. Then $(\mathcal{M}(D),d_{\mathcal{M}(D)})$ is separable.
	\end{enumerate}
	
\end{lemma}
\begin{proof}
	For the first point, we only need to verify the triangle inequality. We take $(X_i)_{i\in \N}$, $(Y_j)_{j\in \N}$ , $(Z_k)_{k\in \N} 
	\in M(D)$. With the definition of $d_{M(D)}$, for every $\epsilon >0$, there are two injective functions $(\lambda ^Y_X : J'\to \N)$ and $(\lambda ^Y_Z : J''\to \N)$, such that
	\begin{align*}
		d_{M(D)}& ((X_i)_{i\in \N},(Y_j)_{j\in \N})\\
		\ge& \sup \{ d (X_{\lambda ^Y_X (j)},Y_j), \forall j\in J'; \frac{\diameter(X_i)}{2}, \forall i\in \N \backslash \lambda^Y_X(J'); \frac{\diameter (Y_j)}{2}, \forall j\in \N \backslash J' \} -\epsilon,\\
		d_{M(D)}& ((Z_k)_{k\in \N},(Y_j)_{j\in \N})\\
		\ge& \sup \{ d (Z_{\lambda ^Y_Z (j)},Y_j), \forall j\in J''; \frac{\diameter(Z_k)}{2}, \forall k\in \N \backslash \lambda^Z_X (J''); \frac{\diameter (Y_j)}{2}, \forall j\in \N \backslash J'' \} -\epsilon.
	\end{align*}
	We define an injective function $(\lambda =\lambda ^Y_Z \circ (\lambda ^Y_X)^{-1} : \lambda ^Y_X (J'\cap J'') \to \lambda ^Y_Z(J'\cap J'') \subset \N)$, where $(\lambda ^Y_X)^{-1}$ is well-defined as we restrict it to $\lambda ^Y_X (J'\cap J'')$. We discuss by cases.
	\begin{itemize}
		\item For $ i\in \lambda ^Y_X (J' \cap J'') $,
		\begin{align*}
			d(X_i, Z_{\lambda (i)}) 
			&\le d(X_i, Y_{(\lambda ^Y_X)^{-1} (i)}) + d(Z_{\lambda (i)}, Y_{(\lambda ^Y_X)^{-1} (i)}) \\
			&\le d_{M(D)}((X_i)_{i\in \mathbb{N}}, (Y_j)_{j\in \mathbb{N}}) 
			+ d_{M(D)}((Z_k)_{k\in \mathbb{N}}, (Y_j)_{j\in \mathbb{N}}) + 2\epsilon;
		\end{align*}
		\item For $i\in \N\backslash \lambda ^Y_X (J')$, $\frac{1}{2}\diameter (X_i) \le d_{M(D)} ((X_i)_{i\in \N},(Y_j)_{j\in \N}) +\epsilon$;
		\item For $i\in \lambda ^Y_X (J') \backslash \lambda ^Y_X (J' \cap J'') $, 
		\begin{align*}
			\frac{1}{2}\diameter (X_i) &\le d (X_i, Y_{(\lambda ^Y_X)^{-1} (i)}) + \frac{1}{2}\diameter (Y_{(\lambda ^Y_X)^{-1} (i)})\\ &\le d_{M(D)} ((X_i)_{i\in \N},(Y_j)_{j\in \N}) + d_{M(D)}((Z_k)_{k\in \N},(Y_j)_{j\in \N})+2\epsilon;
		\end{align*}
		
		\item For $k\in \N \backslash \lambda ^Y_Z( J'')$, $\frac{1}{2}\diameter (Z_k) \le d_{M(D)}((Z_k)_{k\in \N},(Y_j)_{j\in \N})+\epsilon$;
		\item For $k\in \lambda ^Y_Z(J'')\backslash \lambda ^Y_Z (J'\cap J'') $, 
		\begin{align*}
			\frac{1}{2}\diameter (Z_k) &\le d (Z_k, Y_{(\lambda ^Z_X)^{-1} (k)}) +\frac{1}{2}\diameter (Y_{(\lambda ^Z_X)^{-1} (k)})\\ & \le d_{M(D)} ((X_i)_{i\in \N},(Y_j)_{j\in \N}) + d_{M(D)}((Z_k)_{k\in \N},(Y_j)_{j\in \N})+2\epsilon.    
		\end{align*}
	\end{itemize}
	Then we have \[ d_{M(D)} ((X_i)_{i\in \N}, (Z_k)_{k\in \N})) \le d_{M(D)} ((X_i)_{i\in \N},(Y_j)_{j\in \N}) + d_{M(D)}((Z_k)_{k\in \N},(Y_j)_{j\in \N})+2\epsilon.\] 
	The triangle inequality follows by letting $\epsilon$ tend to 0.
	
	For the second point, we consider the embedding $\iota : \bigcup _{n\ge0} \mathscr{L}^n_D \to \mathcal{M}(D)$ defined by
	\[ \iota (l_1,\cdots , l_n) =[l_1,l_2,\cdots, l_n, 0,0, \cdots ], \]
	where $\mathscr{L}_D^0 =\{ \emptyset \}$ and 0 denotes the constant loop based at the point 0. In particular, $\iota (\emptyset)= [0,0,\cdots]$. 
	Note that for every $n\in \N _+$, the restricted function $\iota: (\mathscr{L}_D^n, d\otimes \cdots \otimes d) \to (\mathcal{M}(D), d_{\mathcal{M}(D)})$ is 1-Hölder, combined with the separability of  $ (\mathscr{L}_D,d)$, the image $\iota(\mathscr{L}_D)^n$ is separable. It suffices to prove that the image of $\iota$, $\iota(\bigcup _{n\ge0} \mathscr{L}^n_D)$, is dense in $\mathcal{M}(D)$. For $m \in \mathbb{N}_+$ and $l \in \mathscr{L}_D$, define
	\[
	l^m =
	\begin{cases}
		l, & \text{if } \diameter(l) \ge \frac{1}{m}, \\
		0, & \text{otherwise}.
	\end{cases}
	\]
	Then, for $[l_1,l_2, \cdots]\in \mathcal{M}(D)$, the sequence $([l^m_1, l^m_2,\cdots])_{m\in \N_+} $, which lies in  $\iota(\bigcup _{n\ge0} \mathscr{L}^n_D)$, converges to $[l_1,l_2,\cdots] $. 
	
\end{proof}

We call $(\mathcal{M}(D),d_{\mathcal{M}(D)})$ the \textit{space of loop soups}, and refer to the corresponding topology as the loop soup topology. 
\begin{remark}
	For a fixed mesh size $\delta>0$, the total mass of the random walk loop measure is finite. Then the random walk loop soup $\mathcal{L}^{\#\delta}$ is almost surely finite. In this case, within the space of loop soups, we see that the random walk loop soup lies in the image of the embedding function $\iota$ defined in the proof of Lemma \ref{space of loop soups}. In particular, the empty loop soup corresponds to $\iota(\emptyset)$.
\end{remark}
\begin{remark}
	In this topology, we are only concerned with the macroscopic loops. In fact, as $\delta\to 0$, the total mass of the loop measure tends to infinity, but the total mass of loops with diameter at least $\epsilon$ for a given $\epsilon$ is finite, so the number of macroscopic loops is a.s. finite. If we consider the Brownian loop soup, there will be a.s. infinitely many tiny loops, which we do not care about.
\end{remark}

\subsection{Loop-erased random walk}
Suppose $(X_t)_{0\le t\le T}$ is a random walk on the planar graph $\Gamma$ started from a vertex $v_0$, where $T>0$. While discussing loop erasure, we return to the usual connection and regard the random walk as a c\`adl\`ag process. We denote the left limit at time $s$ by $X(s-)$ or $X_{s-}$. For $0\le s<t$, we denote the trajectory between times $s$ and $t$ by $X^{\#\delta} [s,t]$. 

\begin{definition}[Loop erasure]
	The \textbf{loop erasure} of the random walk $(X_t:0\le t\le T)$ is a self-avoiding path $(Y_n: 0\le n\le S)$ obtained by erasing the loops in chronological order. Precisely, let $Y_0= X_0$, define $T_0 =\sup \{ t \in [0,T] :  X_t =  Y_0 \}$, then let $Y_1= X_{T_0}$. Recursively, for $k\ge 1$, if $T_{k-1} <T$, we define $T_k =\sup \{ t \in [T_{k-1}, T] :  X_t =  Y_k \}$, then let $Y_{k+1} =X_{T_k}$. Continue until $T_{S-1}=T$, define $Y_S= X_T$. We denote $\LE$ to be the loop erasure operation: $\LE(X[0,T])= Y[0,S]$. 
	We call $Y[0,S]$ \textbf{loop erased random walk} (LERW).
\end{definition}
For each $k=0,\cdots, S$, the stopping time $T_k$ is the last time the random walk visits the point $Y(k)$. We denote this time by the last visit operator $\lv(k)$. 
Let $\gamma_1,\gamma_2$ be two curves such that the endpoint of $\gamma_1$ coincides with the starting point of $\gamma_2$. We define their concatenation, denoted by $\gamma_1\oplus \gamma_2$, by connecting these two curves with the common point. Immediately, for $k\in \{ 0, 1,\cdots, S-1\}$, the path $l_k:=X[T_{k-1}, T_k)$ is a loop, and we can decompose the random walk path into the concatenation of the erased loops and the edges in the loop-erased path: 
\begin{equation}\label{LERW}
	X[0,T] = l_0 \oplus Y[0,1] \oplus l_1 \oplus Y[1,2] \oplus \cdots \oplus l_{S-1} \oplus Y[S-1,S].
\end{equation}
\begin{remark}
	There are two types of loops erased in LERW: the first are simple loops removed in chronological order by the erasure procedure; the second are loops rooted at points along the simple curve $Y[0,S]$, namely the loops $X[T_{k-1}, T_k)$ for $k=0,\cdots , S-1$. We refer to the latter as the erased loops in LERW.
\end{remark}

\begin{proposition}[loop addition~\cite{Lawler_Limic_2010}]\label{loop addition}
	Suppose $\gamma=(\gamma_0,\gamma_1,\cdots ,\gamma_k)$ is a self-avoiding path from $v$ to the boundary $\partial D$. Let $(X^v, \mathcal{L}^{\Gamma}_D)$ be an independent pair of random walk from $v$ and a random walk loop soup restricted to $D$. For each $0\le j\le k$, let $\bar{l}_{1,j},\bar{l}_{2,j},\cdots,\bar{l}_{n(j),j}$ denoted the unrooted loops in $\mathcal{L}_D^{\Gamma}$ that intersect with $\gamma_j$ but do not intersect $\{\gamma_0,\cdots,\gamma_{j-1}\}$. The loops are listed in the time order. For each loop $\bar{l}_{i,j}$, we choose uniformly a representative $\tilde{l}_{i,j}$  rooted at $\gamma_j$. Then the rooted loop 
	\[ \tilde{l}_j:= \tilde{l}_{1,j}\oplus \cdots\oplus \tilde{l}_{n(j),j} \]
	has the same law as $l_j$ defined in Equation \eqref{LERW}, conditioned on the event that $LE(X^v)=\gamma$.
	
	In other words, if we add the loops to $\gamma$, we have a path
	\[ \tilde{l}_0 \oplus \gamma[0,1] \oplus \tilde{l}_1 \oplus \gamma[1,2] \oplus \cdots \oplus \tilde{l}_{k-1} \oplus \gamma[k-1,k] \]
	which has the same law with the random walk $X^v$ conditioned  on $LE(X^v)=\gamma$. 
\end{proposition}
A similar version for the Markovian loop soup is stated in Proposition~5.4 of \cite{chang2014markovloopsdiscretespaces}.

\paragraph*{\textbf{Convergence of LERW}}
Lawler, Schramm and Werner~\cite{lawler2003conformalinvarianceplanarlooperased} proved that the scaling limit of a LERW on the square lattice is radial $\SLE_2 $. Yadin and Yehudayoff ~\cite{Yadin_2011} generalized this result to more general planar graphs, assuming an invariance principle. In the chordal case, Suzuki~\cite{suzuki2014convergencelooperasedrandom} proved the convergence of a LERW excursion to chordal $\SLE_2$, and this was further generalized by Uchiyama~\cite{uchiyama2017boundarybehaviourrwsplanar}. 

We define a discrete domain as a union of faces of $\Gamma^{\# \delta}$ along with the edges and vertices incident to them. We say a discrete domain is simply connected if the union of its faces and non-boundary edges and vertices forms a simply connected domain in $\mathbb{C}$.  

Here we refer to Theorem~3.6 and Corollary~3.7 of \cite{berestycki2024dimersriemannsurfacesii} to state the convergence of LERW on the chordal case.
\begin{theorem}[\cite{berestycki2024dimersriemannsurfacesii}]\label{chordal}
	Let $D$ be a simply connected domain such that $\Bar{D}\subset \mathbb{D}$. Let $\Bar{D}^{\#\delta}$ be a sequence of simply connected discrete domains with $D^{\#\delta}$ being its interior. Assume that $\mathbb{C} \backslash D^{\#\delta}$ is uniformly locally connected. Let $p_0\in D$ and suppose that $D^{\#\delta }$ converges in the Carath\'eodory sense to $D$: if $\phi $(resp. $\phi _{\#\delta}$) is the unique conformal map sending $D$(resp. $D_{\#\delta}$) to $\mathbb{D}$ such that $\phi(p_0)=0$ and $\phi '(p_0) >0$(resp. $\phi _{\#\delta}(p_0) =0$ and $\phi _{\#\delta}'(p_0)>0$), then $\phi _{\#\delta}$ converges to $\phi$ uniformly over compact subsets of $D$. Suppose that $a^{\#\delta}$, $b^{\#\delta} $ are two boundary points on $\Bar{D}^{\#\delta} $ such that $\phi _{\#\delta} (a^{\#\delta}) \to \tilde{a}\in \partial \mathbb{D}$ and $\phi _{\#\delta} (b^{\#\delta}) \to \tilde{b}\in \partial \mathbb{D}$ with $\tilde{a}\not= \tilde{b}$.
	
	Let $X^{\#\delta}$ be a random walk subject to the assumptions in Section \ref{graph} from $a^{\#\delta}$ conditioned to take its first step in $D^{\#\delta}$ and to leave $D^{\#\delta}$ at $b^{\#\delta}$. 
	Then $\LE(X^{\#\delta}) $ converges to chordal $\SLE_2$ from $a$ to $b$, where $a=\phi^{-1} (\tilde{a})$ and $b=\phi^{-1}(\tilde{b)}$.
\end{theorem}

\subsection{Wilson's algorithm}\label{section wilson}
We now describe Wilson's algorithm. 
\begin{definition}[Wilson's algorithm~\cite{10.1145/237814.237880}]
	Given the graph $\Gamma$, the boundary $\partial \Gamma$ and an ordering of the vertices $(v_0, v_1, \cdots)$ of $\Gamma$, \textbf{Wilson's algorithm} runs as follows.
	\begin{itemize}
		\item We start from $v_0$ and perform a loop-erased random walk until a boundary vertex is hit.
		\item Recursively, we start from the next vertex in the ordering (we stop if it doesn't exist) which is not included in what we have sampled so far, and perform a loop-erased random walk from it until it hits the boundary or the part of vertices we have sampled before.
	\end{itemize}
\end{definition}
\begin{proposition}[Wilson~\cite{10.1145/237814.237880}]
	Wilson's algorithm generates the wired uniform spanning tree. In particular, the law of the resulting spanning tree is independent of the ordering of vertices.
\end{proposition}

For the k-th LERW performed in Wilson's algorithm, following the decomposition given by Equation \eqref{LERW}, we denote the set of loops erased in the k-th branch by $\mathbb{L}_k = \{ l^k_0, l^k_1, \cdots \}$. Then the set of loops erased in Wilson's algorithm for the given ordering is 
\[ \mathbb{L} =\bigcup _{k\ge 1} \mathbb{L}_k. \]
For $\Gamma=\Gamma\dis $, we write $\mathbb{L}\dis$, $\mathbb{L}\dis_k$ for the corresponding collections.
We remark that the law of $\mathbb{L}^{\#\delta}$ depends on the ordering we choose. In the following, we fix the \textbf{Good Ordering} defined below as the reference ordering of vertices in Wilson's algorithm; for simplicity, we do not indicate the ordering in the notation $\mathbb{L}$. 

\textbf{Good ordering:}
Consider $\{ 6^{-j} \mathbb{Z}^2\}_{j\ge 0}$, a sequence of scaling of the square lattice which divides the plane into square cells. At step $j$, pick a vertex from each cell $f$ of $6^{-j} \mathbb{Z}^2$ which is closest to the center of the cell (if two points have the same distance from the center, we use the lexicographic order) and is not chosen in any previous step. Denote the set of vertices picked in step $j$ by $Q_j$. We order the vertices of $Q_j$ according to the lexicographic order of their corresponding cell. We stop after we have sampled all vertices. We order all the vertices by $(Q_0, Q_1, \cdots)$.

\begin{remark}
	The result in Section \ref{chapter GWA} doesn't depend on the ordering. We introduce the good ordering to estimate the remaining branches, as in Schramm's finiteness theorem (Theorem \ref{theorem schramm}). 
\end{remark}

\paragraph*{\textbf{Relationship between random walk loop soup and Wilson's algorithm}}
We can add the loops in RWLS to an independent LERW by using Proposition \ref{loop addition}, and obtain a random walk. Moreover, given an ordering of vertices, we can iteratively add the loops in RWLS to the branches of a wired uniform spanning tree such that the marginal law of each random walk path is a real random walk performed in Wilson's algorithm. We state it as follows:
\begin{corollary}\label{loop addtion2}
	Given the graph $\Gamma$ and an ordering of vertices $(v_0,v_1, \cdots )$ with any two vertices being different. Let $(\mathcal{L}^{\Gamma}, \mathcal{T})$ be an independent pair of random walk loop soup and the wired uniform spanning tree. Let $(\gamma^0,\gamma^1,\cdots )$ be the self-avoiding path sampled in $\mathcal{T}$. Precisely: Assume that $\gamma^0$ is the branch in $\mathcal{T}$ emanating from $v_0$ to $\partial D$. Recursively, for $k\ge 1$, let $\gamma^k$ be the branch in $\mathcal{T}$ emanating from the first unsampled vertex in $\Gamma\backslash(\cup _{i<k}\gamma^i)$ to $\partial D \cup \gamma^0\cup\cdots \cup \gamma^{k-1}$.
	
	Then we can couple $(\mathcal{L}^{\Gamma}, \mathcal{T})$ with a sequence of random walks  $X^1, X^2,\cdots$  such that:
	\begin{itemize}
		\item For each $k\ge 0$, $LE(X^k)=\gamma^k$.
		\item For each $k\ge 0$, the random walk path $X^k$ is obtained by adding loops in $\mathcal{L}^{\Gamma\backslash(\cup _{i<k}\gamma^i)}$  to $\gamma^k$.
	\end{itemize}
\end{corollary}

\section{Greedy Wilson's algorithm}\label{chapter GWA}
In this section, we study the set of loops erased in Wilson's algorithm. We introduce a variant of Wilson's algorithm called the \emph{Greedy Wilson's Algorithm} (GWA), which is inspired by the algorithm used in \cite{berestycki2024dimersriemannsurfacesii} for sampling a cycle-rooted spanning forest on a graph embedded into a specific compact Riemann surface. While in that paper the authors use the open sets in the chart, we instead use the Euclidean balls of fixed and small radius. We define the loops erased in the GWA, which naturally approximate those in Wilson's algorithm. 
In Section \ref{continuum GWA}, also inspired by \cite{berestycki2024dimersriemannsurfacesii}, we introduce the continuum version of GWA. 
Using the convergence of LERW, we prove that the triple consisting of a self-avoiding path, a local random walk path, and a concatenated loop-erased path in GWA converges in law to that of the continuum version. This, in turn, implies that the set of loops erased in finite branches of the GWA converges to that of the continuum version. See Theorem \ref{scaling limit of XAY} and Corollary \ref{k branches} for more details.

\subsection{Discrete version of Greedy Wilson's algorithm}
In this section, we define the Greedy Wilson's algorithm such that we can ``ignore'' most tiny loops with diameter smaller than $\epsilon$. 

\begin{definition}[Greedy Wilson's algorithm]
	Given $\epsilon >0$, the graph $\Gamma =\Gamma ^{\# \delta }$, the boundary vertices $\partial$ and an ordering of vertices $(v_0, v_1, \cdots)$. Take $r(\epsilon)\in (0,\epsilon)$. Suppose that $\delta <\frac{r(\epsilon)}{100}$. 
	\textbf{The Greedy Wilson's algorithm} is as follows. 
	\begin{enumerate}[label=(\arabic*)]
		\item Let $X^1$ be the simple random walk started from $v_0$ and stopped when hitting the boundary $\partial$. We denote this stopping time by $T^1$ and denote the loop erasure of $X^1$ by $Y^1=Y^1[0, S^1]=\LE(X^1[0, T^1])$. We define a sequence of stopping time $(\tau^1_1,\cdots )$, SRW curves $(X^1_1,\cdots )$, self-avoiding curves $(Y^1_1, \cdots )$ and the portion of the curves $(A^1_1, \cdots)$  as follows.
		\begin{enumerate}[label=(\roman*)]
			\item Define $\tau ^1_1 := \inf \{ t \ge 0: X^1(t) \in \partial \ \text{or} \  X^1(t) \notin B(v_0, \epsilon) \}$. 
			Denote $Y^1_1 =Y^1_1 [0,t^1_1] =\LE (X^1[0,\tau^1_1 ])$ as the loop-erasure up to this point. Define the last visit time to $v_0$ as $\theta ^1_1 := \sup \{ t\in [0,\tau ^1_1]: X^1(t)=v_0\} $, and the corresponding time in $Y^1_1$ as $s^1_1=0$. Denote $X^1_{1}=\{ v_0\}$ and $A^1_1=Y^1_1$.
			\item Inductively, assume that we have defined $\tau^1_k $, $X^1 _k$, $ Y^1_k $ and $ A^1_k$. If  $X^1(\tau^1_k) \in \partial$, we stop the construction of the first branch;
			if $X^1(\tau_k^1)\in \overline{B(v_0,r(\epsilon))}$, we stop the algorithm along this branch, then output ``\textbf{ERROR}'' and the self-avoiding path $Y^1$. 
			
			Otherwise, define $\tau ^1_{k+1} := \inf  \{ t> \tau^1_k:  X^1(t) \in \partial \ \text{or} \  X^1(t) \notin B(X^1_1(\tau^1_k), r(\epsilon)) \} $. 
			Define the self-avoiding path $Y^1_{k+1} =\LE(X^1[0, \tau ^1_{k+1}])=Y^1_{k+1}[0,t^1_{k+1}].$
			Define the last visit point of $X^1[0, \tau^1_{k+1}]$ to $Y^1_k$ as $Y_k^1(s_{k+1}^1)$ where $s^1_{k+1}=\min  \{s\in [0,t^1_k] : Y_{k}^1(s)\in X^1[\tau^1_{k},\tau^1_{k+1}]\}$.
			Define the corresponding last visit time in $X^1[0, \tau^1_{k+1}]$ as $\theta ^1_{k+1} = \sup \{ t\in [\tau ^1_k, \tau ^1 _{k+1}] : X^1_k (t) = Y_k^1(s_{k+1}^1) \} $, and note that $Y^1_k(s^1_{k+1})=X^1(\theta^1_{k+1}-)$.
			Now define the new SRW curve $X^1_{k+1}=X^1[\tau^1 _k, \theta^1_{k+1})$. Define $A^1_{k+1} =\LE (X^1[\theta^1_{k+1},\tau^1_{k+1}])$. A direct result is:
			\[Y^1_{k+1} =Y^1_k[0,s^1_{k+1}]\oplus A^1_{k+1}.\]
		\end{enumerate}
		\item Having defined branches $(Y^1, Y^2, \cdots, Y^m)$, if $Y^1 \cup \cdots \cup Y^m =\Gamma$, i.e. if we have sampled all vertices, we stop; otherwise, we perform step (1) again with $m$ changed into $m+1$, where the simple random walk $X^{m+1}$ starts from the next vertex in the ordering which hasn't been sampled before and stops when hitting the boundary $\partial $ or the branches $Y^1 \cup \cdots \cup Y^m$.
	\end{enumerate}
\end{definition}

\begin{figure}
	\centering
	\includegraphics[width=0.48\linewidth]{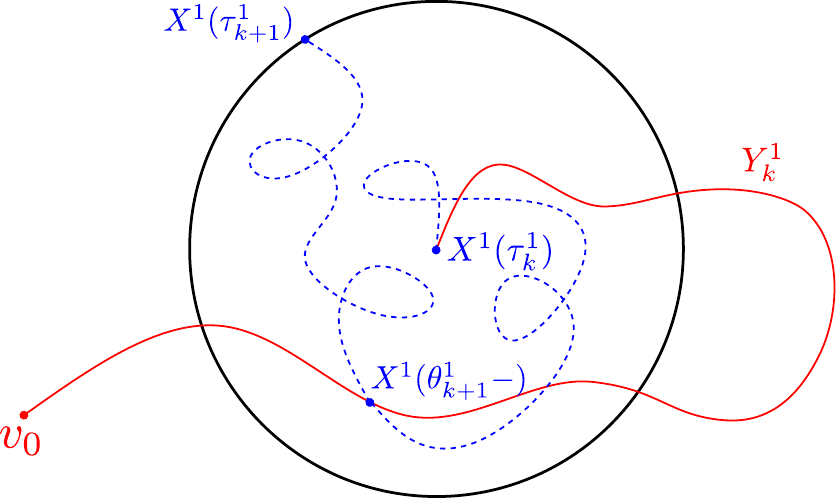}
	\hfill
	\includegraphics[width=0.48\linewidth]{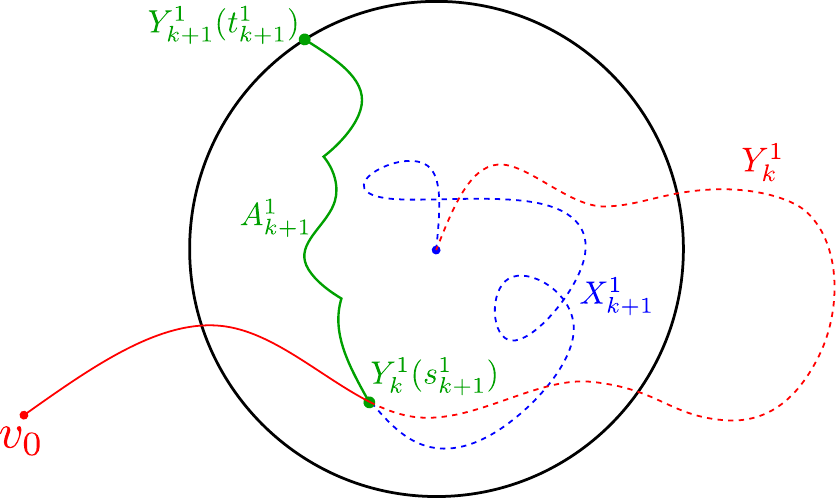}
	\caption{An illustration of the $(k+1)$-th step in the sampling of the first branch of the greedy Wilson's algorithm. Left: The red curve $Y^1_k$ is the loop-erased curve obtained from the first $k$ steps. We run a random walk from $X^1(\tau_k^1)=Y^1_k(t_k)$ until it exits $B(X^1(\tau_k^1), r(\epsilon))$ at the point $X^1(\tau_{k+1}^1)$, shown as the blue curve. The last visit point of the random walk to $Y^1_k$ is $Y^1_k(s^1_{k+1})=X^1(\theta^1_{k+1}-)$. Right:  We perform a loop erasure on the random walk path $X^1[\theta^1_{k+1},\tau_{k+1}^1]$ to obtain the simple curve $A^1_{k+1}$(green curve). The updated simple curve $Y_{k+1}^1$ is obtained by concatenating $Y^1_k[0,s^1_{k+1}]$ (solid red line) and  $A^1_{k+1}$ (solid green line). The dashed blue curve represents $X^1_{k+1}$. }
\end{figure}

\begin{remark}
	Let us stress that when sampling each branch in the algorithm described above, the first step uses a ball of radius $\epsilon$, while the subsequent steps use balls of radius $r(\epsilon)$. This choice ensures that, with high probability, after the first step, the random walk does not return too close to its starting point, and that the connected component of $B(X(\tau_k), r(\epsilon))\backslash (Y_k[0,s_{k+1}]\cup \partial)$ is simply connected, since this open set excludes the starting point. The simply connectivity is crucial in the continuum setting to define the chordal $\SLE_2$. 
	We also note that in \cite{berestycki2024dimersriemannsurfacesii}, there is a mistake in the argument claiming that $D_k$ is simply connected (specifically in item (iii) of the definition of the continuum version of Wilson's algorithm to generate CRSF). However, this issue can be solved by using the same strategy: setting the initial neighborhood radius to be significantly larger than those of subsequent steps.
\end{remark}

For the $m$-th branch $Y^m$, we count the number of iterations as a random variable  $N^m $ such that $\tau^m _{N^m}=T^m$. Immediately, $N^m$ is a.s. finite. Moreover, we can prove that $N^m$ has an exponential tail.

\begin{lemma}[Exponential tail of the number of iterations]\label{exp tail} 
	For $0<r(\epsilon) \le \epsilon $, there exist $\alpha (\epsilon), \beta(\epsilon)>0$, depending on $\epsilon$ and $r(\epsilon)$, such that, for $\delta$ small enough, for every $m\ge 0$, for every $K>2\diameter(D)/\epsilon$,
	\[ \mathbb{P} (N^{m,\#\delta}\ge K) \le \beta(\epsilon) \exp{(-\alpha (\epsilon) K)}. \]
\end{lemma}

\begin{proof}
	Let $(B_t)_{t\ge 0}$ be the standard 2D Brownian motion. Let $\tau_{r(\epsilon)}$ be the hitting time of $\partial B(0,r(\epsilon) )$. 
	Then $ \mathbb{P}^0 ( \Re( B_{\tau _{r(\epsilon)} }) > \frac{\sqrt{3}}{2}r(\epsilon) ) =\frac{1}{6}$. 
	By the invariance principle and the Markov property, for $\epsilon>0$, for $\delta$ small enough, for each vertex $v^{\#\delta} $, let $\tau _{r(\epsilon)}^{v}$ be the time that the random walk starting from $v^{\#\delta }$ leaves $B^{\#\delta}(v^{\#\delta},r(\epsilon))$, then
	\[\mathbb{P}^v( \Re \{v^{\#\delta} -X^{\#\delta} (\tau^v_{r(\epsilon)}) \} >\frac{r(\epsilon)}{2}   ) \ge \frac{1}{6}. \]
	Then the probability for the random walk to exit $D$ within $2\diameter (D) /r(\epsilon)$ iterations is at least $6^{- 2\diameter (D) /r(\epsilon) }$. 
	Take $\beta(\epsilon)=(1-6^{-2\diameter(D)/r(\epsilon)})^{-1} $, take $\alpha (\epsilon ) =\frac{-r(\epsilon)}{2\diameter(D)} \ln (1-6^{-2\diameter(D)/r(\epsilon)})$. For $K > 2\diameter (D) /r(\epsilon)$, we have
	\begin{align*}
		\mathbb{P} ( N \ge K) \le \left(1-6^{-2\diameter(D)/r(\epsilon)}\right)^{\frac{r(\epsilon) K}{2\diameter(D)}-1}=\beta(\epsilon) \exp(-\alpha(\epsilon)K).  
	\end{align*}
	We apply the above estimate to every branch.
\end{proof}

We actually want to see, for every branch, $k\mapsto \LE(X^s [0, \tau^s_k])=Y^s_k$ as a Markov chain, and in the next subsection, we will provide a continuum version of this chain.  To approximate the decomposition Equation \eqref{LERW}, we define the ``large'' loops erased in this algorithm. For simplicity, we will consider only one branch and ignore the superscripts. We define the set of integers:
\[ R= \left\{ s\in \{ 0,\cdots ,S\} \mid \exists n\in \{0,\cdots ,N\}, \text{such that} \ s=s_n,\  Y[0,s]=Y_n [0,s_n] \right\}.\]
Notice that for $s\in \{ 0,\cdots ,S\}$, if $\diameter (l_s) \ge 2\epsilon$, after sampling $Y[0,s-1]$ in this algorithm, $l_s$ intersects with at least two balls with radius $r(\epsilon)$, so $Y(s)$ has to be touched twice in the greedy Wilson's algorithm, i.e. $s\in R$.  We list the elements in increasing order, $R=\{ \xi_1,\cdots , \xi _r\}$.  Based on this observation, for a given branch of the greedy Wilson's algorithm, we define the erased loops, $(\tilde{l_s} : s\in \{0,\cdots , S-1\})$, as follows. 
\begin{enumerate}
	\item In the event that no ERROR occurs in this branch, we define the erased loops as follows.
	\begin{itemize}
		\item For $s\notin R$, we denote $\tilde{l_s}=Y(s)$.
		\item For $s=\xi_1$, we define the starting point and ending point:
		\[ n_{\xi _1}^+ =\max \{ n\in \{ 0,\cdots , N\}: \xi_1=s_n\} ; \]
		\[n^- _{\xi_1} = \min \{ n\in \{ 0,\cdots , N\}: Y[0,\xi_1] \subset Y_n\}.  \]
		Then we define 
		\[ \tilde{l}_{\xi_1} := (Y_{n^-_{\xi_1}} \backslash Y[0,\xi_1)) \oplus \left(\bigoplus_{n=n^-_{\xi_1}+1}^{n^+_{\xi_1}-1}(X_n \oplus A_n)\right) \oplus X_{n^+_{\xi_1}}. \]
		\item Inductively, suppose we have defined $n_{\xi _k}^+$, $n_{\xi _k}^-$ and $\tilde{l}_{\xi_k}$.
		\begin{itemize}
			\item If $\{ n\in \{ n^+_{\xi_k} ,\cdots, N\} : \xi_{k+1} =s_n\}=\emptyset$, let $\tilde{l}_{\xi_{k+1}} = Y(\xi _{k+1})$ and $$n^+_{\xi_{k+1}}=n^-_{\xi_{k+1}}=\min \{ n\ge n_{\xi _k}^+ : Y[0,\xi_{k+1}] \subset Y_n\}.$$
			\item Otherwise, we define 
			\[ n^+_{\xi_{k+1}} = \max \{ n\in \{ n^+_{\xi_k} ,\cdots, N\} : \xi_{k+1} =s_n\}, \]
			\[ n^-_{\xi_{k+1}} =\min \{ n\in \{ n^+_{\xi_k},\cdots , N\}: Y[0,\xi_{k+1}] \subset Y_n\}, \]
			and the loop
			\begin{equation}\label{greedy loop}
				\tilde{l}_{\xi_{k+1}} := (Y_{n^-_{\xi_{k+1}}} \backslash Y[0,\xi_{k+1})) \oplus \left(\bigoplus_{n=n^-_{\xi_{k+1}}+1}^{n^+_{\xi_{k+1}}-1}(X_n \oplus A_n)\right) \oplus X_{n^+_{\xi_{k+1}}}. 
			\end{equation}
		\end{itemize}
	\end{itemize}
	\item In the event that the ERROR occurs in this branch, define $(\tilde{l_s} : s\in \{0,\cdots , S-1\})$ as trivial loops.
\end{enumerate}
Let $\kappa$ denote the index of the first branch in which ERROR occurs. We then define \textbf{the set of loops erased in the first $m$ branches of the greedy Wilson's algorithm} as the collection of loops erased in the first $\min\{\kappa ,m\}$ branches, denoted by $\mathbb{L}_{\epsilon,m} ^{\#\delta}$.

Immediately, we can prove the following lemma. In particular, the third point shows that the loops erased in GWA are naturally coupled with those in Wilson's algorithm, with a distance of order $\epsilon$.
\begin{lemma}
	Given a branch $Y[0,S]=\LE(X[0,T])$, for $s\in \{0,\cdots , S-1\}$,  we denote the last visit time of $Y(s)$ by $\lv(s) :=\sup \{t\in [0,T] \mid X(t) =Y(s)\}$. Conditioned on the event that no ERROR occurs, we have the following:
	\begin{enumerate}[label=(\roman*)]
		\item For $s\notin R$, assume that $s\in [\xi_k,\xi_{k+1}]$ for an integer $k$ with the convention that $\xi_0 =0$, and let $n_s=\inf \{ n\in \{ n^+_{\xi_k} , \cdots , n^-_{\xi_{k+1}}\} \mid Y[0,s] \subset Y_n\}$. Then
		\[ \tau_{n_s-1}  \le \lv(s-1) \le \lv(s) \le \tau_{n_s} .\]
		\item For $s\in R$, we have
		\[ \tau_{n^+_s -1} \le \lv(s)\le \tau_{n^+_s},\]
		\[ \tau_{n^-_s-1} \le \lv (s-1)\le \tau _{n^-_s}. \]
		\item For $s\in \{0,\cdots, S-1\}$, $d (l_s ,\tilde{l}_s) \le 2\epsilon$.
	\end{enumerate}
\end{lemma}
\begin{proof} This mainly follows from the definition.
	\begin{enumerate}[label=(\roman*)]
		\item By definition, the index of the vertex in $Y_{n_s}$ reached by the random walk after time $\tau_{n_s}$ is greater than $s$, i.e. $\lv(s) \le \tau_{n_s}$. Also, since $Y(s) \in A_{n_s} $ and $s\notin R$, $Y(s)$ is not the starting point of $A_{n_s}$, then $Y(s-1) \in B(Y_{n_s-1} (\tau_{n_s-1}),\epsilon) $, i.e. $\lv(s-1)\ge \tau_{n_s-1}$.
		\item The first inequality follows from the definition. The second inequality follows  from $n^-_{s}=n^+_{s-1}$. 
		\item This follows from the two points mentioned above. If $s\notin \mathcal{R}$, the loop $\tilde{l}_s=Y(s)$ is constant. It follows from (i) that the loop $l_s=X[\lv(s-1),\lv(s))$ is contained in the ball $B(Y_{n_s-1} (\tau_{n_s-1}),\epsilon)$. Then we have $d (l_s ,\tilde{l}_s) \le 2\epsilon$. 
		Conversely, if $s\in \mathcal{R}$, we decompose the loop $l_s$ into several segments as follows:
		\begin{align*}
			l_s=X[\lv(s-1),\lv(s)) = X[\lv (s-1), \tau_{n_s^-}) \oplus 
			\left(\bigoplus_{n=n^-_s+1}^{n^+_s-1}(X_n \oplus X[\theta_n,\tau_n))\right) \oplus X_{n^+_s}.
		\end{align*}
		In comparison with the concatenation definition of $\tilde{l}_s$ in Equation \eqref{greedy loop}, together with (2), we conclude that $d (l_s ,\tilde{l}_s) \le 2\epsilon$. 
	\end{enumerate} 
\end{proof}

\subsection{Continuum version of greedy Wilson's algorithm}\label{continuum GWA}
Since the random walk path converges to Brownian motion and the LERW converges to chordal/racial $\SLE_2$,  we can define the \textbf{continuum version of greedy Wilson's algorithm for one branch}: Given a simply connected domain $D$, its boundary $\partial$ and a starting point $z_0\in D$.
\begin{enumerate}
	\item We start a Brownian motion from $z_0$ until it intersects $\partial \cup \partial B(z_0, \epsilon)$, and denote the stopping point by $z_1$. Let $\mathcal{Y}_1$ be a radial $\SLE_2$ in $B(z_0,\epsilon)$ from $z_0$ to $ z_1$. Let $\mathcal{X}_1=\{z_0\}$ and $\mathcal{A}_1 =\mathcal{Y}_1$.
	\item Inductively, assume that we have defined the continuum curves up to step $k$ and call $\mathcal{Y}_k$. Let $z_k$ be the endpoint of $\mathcal{Y}_k$. If $z_k \in \partial$, we stop the algorithm and output $\mathcal{Y}_k$. If $z_k\in \overline{B(z_0,r(\epsilon))}$, we stop the algorithm and output ``ERROR''.\\
	Otherwise, we do the following.
	We start a standard Brownian motion independent of everything else from $z_k$ until it intersects $\partial  \cup \partial B(z_k,r(\epsilon))$. Denote the point where we stop the Brownian motion by $z_{k+1}$. Let $V_{k+1}$ be the infimum of the set of points in $\mathcal{Y}_k$ which the Brownian motion intersects. We denote by $\mathcal{X}_{k+1}$ the portion of Brownian motion from $z_k$ up to its last visit to $V_{k+1}$. 
	Denote the portion of $\mathcal{Y}_k$ from $z_0$ to $V_{k+1}$ by $\mathcal{Z}_k$. Now we define a chordal $\SLE_2$ curve $\mathcal{A}_{k+1}$ in the continuous component containing  $V_{k+1}$ of $B(z_k ,r(\epsilon)) \backslash \mathcal{Z} _k $ from $V_{k+1}$ to $z_{k+1}$. Define
	\[ \mathcal{Y}_{k+1} :=\mathcal{Z} _k \oplus \mathcal{A}_{k+1}. \]
\end{enumerate}

To sample finitely many branches, from points $x_1,x_2,\cdots$, we iterate this algorithm by updating $\partial$ to include all the branches that we have sampled and choose the new starting point.

With the strong Markov property, we can glue these Brownian motions to form a standard Brownian motion, so the number of iterations of this induction, denoted by $\mathcal{N}$, is a.s. finite. By the same proof as Lemma \ref{exp tail}, $\mathcal{N}$ has an exponential tail with the same parameters.  Let $\mathcal{Y}_n=\mathcal{Y}_{\mathcal{N}}$ for $n>\mathcal{N}$. Then we can define a limiting path $\mathcal{Y}_{\infty} := \lim_{k\to \infty} \mathcal{Y}_k =\mathcal{Y}_{\mathcal{N}}$, which is the branch that we sample. 

\begin{theorem}\label{scaling limit of XAY}
	Fix $\epsilon >0$. Let $\partial ^{\#\delta}$ be a set of edges in $\Gamma^{\#\delta}$ and assume that $\partial ^{\#\delta}$ converges in the Hausdorff sense to some set $\partial \subset \overline{D}$. Assume $\partial D\subset \partial$. Assume $D\backslash \partial$ is simply connected and $\mathbb{C}\backslash (D\backslash \partial)$ is locally connected. Assume that $v_0\dis \in \Gamma\dis$ converges to $z_0\in D$. Then, for any $k\ge 1$, we have convergence in law, in the sense of the product topology on $(\mathcal{C}(\bar{D}),d_{curve})^{3k}$:
	\begin{equation*}
		(X^{\#\delta}_n, A^{\#\delta}_n, Y^{\#\delta}_n)_{1\le n\le k}  \xrightarrow[\delta \to 0]{(d)} ( \mathcal{X}_n, \mathcal{A}_n, \mathcal{Y}_n )_{1\le n\le k}.
	\end{equation*}
	Furthermore, in the event that an ``ERROR'' occurs, we adopt the convention that the sequence of triples remains constant after the step at which the ``ERROR'' occurs.
\end{theorem}

\begin{proof}
	The proof closely follows that of Theorem~3.10 in \cite{berestycki2024dimersriemannsurfacesii}. The difference is that we keep track of $\mathcal{X}_i$ and $\mathcal{A}_i$ here, and we account for the ``ERROR'' event. For completeness, we give a self-contained proof.
	
	We prove this theorem by induction and prove at the same time that $\mathcal{Y}_k$ is a simple curve a.s. and the convergence in law of $X^{\#\delta}(\tau_k)$ to $z_k$ at each step k. We use the notations in the description of the greedy Wilson's algorithm and its continuum version. In the first step, by definition, $X_1^{\#\delta}$ and $\mathcal{X}_1$ are just singleton sets. By the invariance principle, $X^{\#\delta}_1(\tau_k)$ converges in law to $z_1$. The convergence of $Y^{\#\delta} _1=A_1^{\#\delta}$ to $\mathcal{Y}_1=\mathcal{A}_1$ is just an application of the result of Yadin and Yehudayoff ~\cite{Yadin_2011}. Since $\mathcal{Y}_1$ is a $\SLE_2$ type curve, it is in particular a.s. a simple curve.
	
	Assume that we have proved the result for $k\in \mathbb{N}$. Since the indicator of the event $\{ |X^{\#\delta}(\tau_k)-X^{\#\delta}(0)|\le r(\epsilon)\}$ converges in law to that of $\{ |z_k-z_0|\le r(\epsilon)\}$, the indicator of the event ``ERROR'' occurs within the first $k$ steps converges to that in the continuum. The result is trivial if an ``ERROR'' has already occurred; thus, we assume that no ``ERROR'' occurs up to step $k$.
	
	We consider $B(X^{\#\delta}(\tau^{\#\delta}_k),r(\epsilon))$ and $B(z_k,r(\epsilon))$. By the invariance principle, we couple the random walk $X^{\#\delta}(t)$ and a reparametrized Brownian motion $B_{\phi(t)}$ such that they are uniformly close in compact time intervals. On the one hand, let $\phi(\tau)$ be the time at which the Brownian motion hits $\partial \cup \partial B(z_k,r(\epsilon) )$, by the Beurling estimate Lemma \ref{beurling estimate}, for $\delta$ small enough, the random walk will, with high probability, intersect with $\partial^{\#\delta} \cup \partial B^{\#\delta}(X^{\#\delta}(\tau^{\#\delta}_k),r(\epsilon))$ after time $\tau$ at a position close to $B_{\phi(\tau)}$. On the other hand, applying the Beurling estimate for the Brownian motion at time $\tau_{k+1}^{\#\delta}$, we can see that the Brownian motion will next hit $\partial \cup \partial B(z_k,r(\epsilon))$ at a nearby point. Then $X^{\#\delta}(\tau^{\#\delta} _{k+1})$ converges in law to $z_{k+1}$ as $\delta \to 0$. 
	
	For the convergence of the last visit point of the random walk $X^{\#\delta}[\tau^{\#\delta}_k, \tau^{\#\delta}_{k+1}]$ to $Y_k^{\#\delta}$, $X^{\#\delta} (\theta^{\#\delta}_{k+1})$, and the random walk path $X^{\#\delta}_{k+1}$, we consider the function $f:\mathcal{C}\times\mathcal{C}\to D\times \mathcal{C}$ which maps two curves $(\eta[a,b],\gamma[b,c])$ to $(x, \gamma')$ where $x$ is the last visit point of $\gamma$ to $\eta$, and $\gamma'$ is the portion of $\gamma$ from $b$ up to its last visit to $x$. In particular, 
	\[ f(Y^{\#\delta} _k, X^{\#\delta}[\tau^{\#\delta}_k, \tau^{\#\delta}_{k+1}])=(X^{\#\delta} (\theta^{\#\delta}_{k+1}),X^{\#\delta}_{k+1});
	f( \mathcal{Y}_k,B^{z_k}[z_k,z_{k+1}])=(V_{k+1},\mathcal{X}_{k+1}). \]
	Notice that the curve $B^{z_k}[z_k,z_{k+1}]$ is independent of $\mathcal{Y}_k$ by the Markov property and acts as a planar Brownian motion. By the local infinite winding property of planar Brownian motion, the closed set $B^{z_k}[z_k,z_{k+1}]\cap \mathcal{Y}_k$ almost surely contains a countable dense subset of topological crossing points. Then the last visit point can be approximated by crossing points in this countable set, which shows that $f$ is a.s. continuous with respect to the first coordinate. Moreover, using the same strategy and the definition of the last visit point, $f$ is a.s. continuous with respect to the second coordinate. Hence $f$ is a.s. continuous, and the pair $(X^{\#\delta} (\theta^{\#\delta}_{k+1}),X^{\#\delta}_{k+1})$ converges in law to $(V_{k+1},\mathcal{X}_{k+1})$. Also, $\LE(X^{\#\delta}[0,\theta^{\#\delta}_{k+1}])$ converges in law to $\mathcal{Z}_{k}$.
	
	Conditioned on $X^{\#\delta}[\tau^{\#\delta}_k, \theta^{\#\delta}_{k+1}]$, $Y_k^{\#\delta}$ and $X\dis(\tau^{\#\delta}_{k+1})$, the law of $X\dis[\theta^{\#\delta}_{k+1},\tau ^{\#\delta}_{k+1}]$ is the same as a simple random walk starting from $X^{\#\delta}(\theta^{\#\delta}_{k+1})$ conditioned to exit $B(X^{\#\delta}(\tau^{\#\delta}_k),r(\epsilon)) \backslash Y^{\#\delta}[0,s_{k+1}]$ at $X^{\#\delta}(\tau^{\#\delta}_{k+1})$. We denote the continuous component of $B(X\dis(\tau^{\#\delta}_k),r(\epsilon))^{\#\delta} \backslash Y^{\#\delta}[0,s_{k+1}]$ containing $X^{\#\delta}(\theta^{\#\delta}_{k+1})$ by $D^{\#\delta}_{k+1}$. For the continuum version, similarly, we denote the continuous component of $B(z_k,\epsilon) \backslash \mathcal{Y}[z_0, V_{k+1}]$ by $D_{k+1}$. So $D^{\#\delta}_{k+1}$ and $D_{k+1}$ are simply connected.  
	
	By the induction hypothesis and Skorokhod's representation theorem, without loss of generality, suppose that the convergence of  $Y_k^{\#\delta }$, $X^{\#\delta} (\tau^{\#\delta}_k)$ and $X^{\#\delta}(\theta^{\#\delta}_{k+1})$ holds almost surely. Hence $\partial D_{k+1}^{\#\delta}$ converges almost surely to $ \partial D_{k+1}$ in the Hausdorff sense. Combined with the simply connectivity of $D_{k+1}\dis$ and $D_{k+1}$, this implies that $D^{\#\delta}_{k+1}$ converges in the sense of kernel convergence(\cite[Section 1.4]{pommerenke2013boundary}). By the Carath\'eodory kernel theorem(\cite[Theorem 1.8]{pommerenke2013boundary}), this implies the convergence of $D^{\#\delta}_{k+1}$ in the Carath\'eodory sense. Since $\mathcal{Z}_{k}$ is simple curve, $\mathbb{C}\backslash D^{\#\delta}_{k+1}$ is uniformly locally connected. Hence, by applying Theorem \ref{chordal}, $A_{k+1}^{\#\delta}$ converges to $\mathcal{A}_{k+1}$. 
	Since $\mathcal{Y}_{k+1}=\mathcal{Z}_k \cup \mathcal{A}_{k+1}$, $Y_{k+1}^{\#\delta} $  converges to $\mathcal{Y}_{k+1}$. Consequently, we obtain the convergence of the joint law.
\end{proof}

Since the number of iterations $\mathcal{N}$ is a.s. finite, the set of intersecting points, i.e. $\mathcal{R}=\{ z \in \mathcal{Y}_{\infty} \mid \exists n\in [1,\mathcal{N}]\  \text{ s.t. } \ z=V_n \}$, is a.s. finite. For a continuous curve $\mathcal{Y}$ and two points $a, b$, we denote the portion of $\mathcal{Y}$ between $a$ and $b$ by $\mathcal{Y}[a,b]$. Now we define the loops erased in the continuum version by induction.
\begin{enumerate}
	\item We take the infimum of $\mathcal{R}$, denoted by $s_1$. Similar to the discrete version, we define the starting point and the ending point:
	\[ n^+_{s_1}=\max \{ n\in \{1,\cdots , \mathcal{N}\} : s_1= V_n \}; \]
	\[ n^-_{s_1}=\min \{ n\in \{1,\cdots , \mathcal{N}\} : s_1\in \mathcal{A}_n \}. \]
	Then we define the loop:
	\[ l^c_{s_1 } := \mathcal{A} _{n^-_{s_1}} [s_1,V_{n^-_{s_1}}]\oplus \left( \bigoplus_{n=n^-_{s_1}+1}^{n^+_{s_1}-1} (\mathcal{X}_n\oplus \mathcal{A}_n) \right) \oplus \mathcal{X}_{n^+_{s_1}}. \]
	\item Inductively, suppose we have defined $s_k$, $n^+_{s_k}$, $n^-_{s_k}$ and $l^c_{s_k}$. If $n^+_{s_k}=\mathcal{N}$, we stop. Otherwise, we take the infimum of $(\mathcal{Y}_{\infty}(s_k, z_{\mathcal{N}}) \backslash \{ s_k\}) \cap \mathcal{R}  $, noted by $s_{k+1}$, then define the starting point and ending point:
	\[ n^+_{s_{k+1}}=\max \{ n\in \{n^+_{s_k}+1,\cdots , \mathcal{N}\} : s_{k+1}= V_n \}; \]
	\[ n^-_{s_{k+1}}=\min \{ n\in \{n^+_{s_k}+1,\cdots , \mathcal{N}\} : s_{k+1}\in \mathcal{A}_n \}. \]
	Then we define the loop:
	\begin{equation}\label{continuum greedy loop}
		l^c_{s_{k+1}} := \mathcal{A} _{n^-_{s_{k+1}}} [s_{k+1},V_{n^-_{s_{k+1}}}]\oplus \left( \bigoplus_{n=n^-_{s_{k+1}}+1}^{n^+_{s_{k+1}}-1} (\mathcal{X}_n\oplus \mathcal{A}_n) \right) \oplus \mathcal{X}_{n^+_{s_{k+1}}}.
	\end{equation}
\end{enumerate}

We run the algorithm for an ordering of starting points $x_1, x_2, \cdots,x_m$.
Let $\varkappa$ denote the index of the first branch in which ERROR occurs. We then define \textbf{the set of loops erased in the first $m$ branches of the continuum version of the greedy Wilson's algorithm} as the collection of the loops erased in the first $\min\{\varkappa ,m\}$ branches, denoted by $\mathbb{L}^c_{\epsilon,m}$. We have the following corollary.

\begin{corollary}\label{k branches}
	Fix $\epsilon >0$. Let $\partial ^{\#\delta}$ be a set of edges in $\Gamma^{\#\delta}$ and assume that $\partial ^{\#\delta}$ converges in the Hausdorff sense to the boundary $\partial D$. Assume that $\mathbb{C}\backslash D$ is locally connected.
	We choose $x_1^{\#\delta},\ x_2^{\#\delta}, \cdots , x_m^{\#\delta} \in \Gamma\dis \cap D$ and $x_1,\cdots, x_m \in D$, such that for every $i$, $x_i^{\# \delta} $ converges to $x_i$ as $\delta \to 0$. Then, in the sense of the loop soup topology, we have the convergence in law:
	\[ \mathbb{L}_{\epsilon, m}^{\#\delta} \xrightarrow[\delta \to 0]{(d)} \mathbb{L}_{\epsilon, m}^c.  \]
\end{corollary}

\begin{proof}
	For any $j\in \mathbb{N}$, let $\mathcal{T}_j^{\#\delta}$ and $\mathcal{T}_j$ be the union of the first $j$ branches generated by Wilson's algorithm. Since $\partial^{\#\delta} \cup \mathcal{T}_j^{\#\delta}$ converges in law in the Hausdorff sense to $D\backslash \mathcal{T}_j$, and $\mathbb{C}\backslash (D\backslash \mathcal{T}_j)$ is locally connected, it suffices to prove the result for the first branch. For $k\in \mathbb{N}^*$, we apply the loop construction to $(X^{\#\delta}_n, A^{\#\delta}_n, Y^{\#\delta}_n)_{1\le n\le k}$ and $( \mathcal{X}_n, \mathcal{A}_n, \mathcal{Y}_n )_{1\le n\le k}$ and denote them by $\mathbb{L}^{\#\delta} (\epsilon, k)$ and $\mathbb{L}^c(\epsilon,k)$ (similarly, if ``ERROR'' occurs in these segments, we set the collection of loops to be trivial). Immediately, by the definition of the numbers of iterations $N\dis$ and $\mathcal{N}$, we have $\mathbb{L}^{\#\delta}_{\epsilon,1} =\mathbb{L}^{\#\delta}(\epsilon, N\dis)$ and $\mathbb{L}^c_{\epsilon,1} =\mathbb{L}^c (\epsilon, \mathcal{N})$.
	
	Let $\mathcal{E}(k)$(resp. $\mathcal{E}(k)\dis$) be the event that ERROR occurs in these $k$ segments of the continuum(resp. discrete) version. 
	Observe that the function $(z_n)_{0\le n\le k} \mapsto \mathbb{1}_{\mathcal{E}(k)}= \mathbb{1}_{\cup_{n=1}^k\{ z_n\in \overline{B(z_0,r(\epsilon))}\}}$ is almost surely continuous, we have the convergence in law:  $\mathbb{1}_{\mathcal{E}(k)\dis} \xrightarrow[\delta \to 0]{(d)} \mathbb{1}_{\mathcal{E}(k)}$.
	In the proof of theorem \ref{scaling limit of XAY}, we have proved that $( \mathcal{X}_n, \mathcal{A}_n, \mathcal{Y}_n )_{1\le n\le k} \mapsto (V_n)_{1\le n\le k}$ is a.s. continuous. Since $\mathcal{Y}_n$ is almost surely a simple curve, the function that decomposes it into two parts at a given point is continuous. Combined with the definition of the set of loops (see Equations \eqref{greedy loop} and \eqref{continuum greedy loop}), the set of erased loops $\mathbb{L}^{\#\delta} (\epsilon,k) $ converges in law to $\mathbb{L}^c(\epsilon,k)$.  Since $\mathcal{N}$ is almost surely finite and has an exponential tail, by using the Lévy-Prokhorov metric associated with $\mathcal{M}(D)$,
	\begin{align*}
		d_{LP} (\mathbb{L}^{\#\delta}_{\epsilon,1},\mathbb{L}^c_{{\epsilon, 1}}) &\le d_{LP} (\mathbb{L}^{\#\delta}_{\epsilon,1}, \mathbb{L}^{\#\delta}(\epsilon,k))+
		d_{LP}(\mathbb{L}^{\#\delta}(\epsilon,k),\mathbb{L}^c(\epsilon, k))+ d_{LP}(\mathbb{L}^c(\epsilon,k), \mathbb{L}^c_{\epsilon,1})\\
		&\le \mathbb{P}(N^{\#\delta}>k)+ d_{LP}(\mathbb{L}^{\#\delta}(\epsilon,k),\mathbb{L}^c(\epsilon, k)) + \mathbb{P}(\mathcal{N}>k)\\
		&\le 2\beta(\epsilon)\exp{(-\alpha(\epsilon ) k)}+d_{LP}(\mathbb{L}^{\#\delta}(\epsilon,k),\mathbb{L}^c(\epsilon, k)) .
	\end{align*}
	We conclude by first taking $k \to +\infty$ and then $\delta\to 0$.
\end{proof}

\section{Scaling limit of random walk loop soup}\label{chapter proof}
In this section, we complete the proof of the main theorem \ref{scaling limit}. Recall from the introduction that our strategy is to compare the loop soup with the set of loops erased in the greedy Wilson's algorithm. At this stage, we have identified the scaling limit for the set of loops erased during the sampling of the first $m$ branches of the greedy Wilson's algorithm, conditioned on no ERROR occurring. To conclude, we need two final elements. First, we need to say that all remaining branches after the $m$ first are unlikely to contribute--this is the goal of Section \ref{section schramm's finiteness lemma}. Second, we need to deal with the extra splitting from the loops erased in Wilson's algorithm, which is shown in Section \ref{Concatenation probability} to be unnecessary in the limit.

\subsection{Schramm's finiteness theorem}\label{section schramm's finiteness lemma}
Schramm's finiteness theorem was first introduced in \cite{schramm1999scalinglimitslooperasedrandom}, and later generalized in \cite{berestycki2018dimersimaginarygeometry} to graphs satisfying bounded density and the uniform crossing estimate. Roughly, after sampling a finite number of branches according to a special ordering of the starting points, with high probability, the random walks emanating from the remaining points and their loop erasure (i.e., all remaining branches) have diameter at most $\epsilon$. Consequently, with high probability, no large loops will be erased after sampling finite branches. The proof closely follows that of Lemma 4.18 of \cite{berestycki2018dimersimaginarygeometry}; the main difference is that Lemma 4.18 in \cite{berestycki2018dimersimaginarygeometry} only samples the tree close to a specific point, while we sample the whole tree, as in \cite{schramm1999scalinglimitslooperasedrandom}. For completeness, we provide a proof in Appendix \ref{appendix B}.

\begin{theorem}[Schramm's finiteness theorem~\cite{berestycki2018dimersimaginarygeometry}]\label{theorem schramm}
	Fix $\epsilon >0$. Let $\delta_0$ be the uniform constant in the uniform crossing assumption. Assume Wilson's algorithm is run according to the Good Ordering defined in Section \ref{section wilson}. Then there exists an integer $j_0=j_0(\epsilon)$ depending only on $\epsilon$, such that for all $j >j_0$ and all $\delta < 6^{-j_0} \delta_0$, the following event holds with probability at least $1-\epsilon$, the random walks emanating from all vertices after sampling $j_0$ branches have diameter at most $\epsilon$.
	
	An immediate corollary is that: if we use the Lévy-Prokhorov metric associated with $\mathcal{M}(D)$, denote by $d_{LP}$, then for all  $\delta < 6^{-j_0} \delta_0$ and all $j>j_0$, 
	\begin{equation*}
		d_{LP} (\mathcal{L}^{\#\delta }, \mathcal{L}^{\#\delta}_{j}) \le \epsilon.
	\end{equation*}
\end{theorem}

The following corollary shows that the mass of the macroscopic loops close to the boundary is arbitrarily small.
\begin{corollary}\label{boundary}
	For all $\epsilon>0$, there exists $\eta>0$, such that for all sufficiently small $\delta$, the following event holds with probability at least $1-\epsilon$: no loop in $\mathcal{L}\dis_D$ has distance less than $\eta$ from $\partial D$ and diameter larger than $\epsilon$. Moreover, for all $\epsilon>0$, there exists $\eta'>0$, such that for sufficiently small $\delta$,
	\[ \mu^{\#\delta}_D (\{ l\in \mathscr{L}_D: \diameter (l)\ge \epsilon; l\cap (\partial D)^{\eta'}\not= \emptyset \}) \le \epsilon, \]
	where $(\partial D)^{\eta'}$ denotes the set of points at distance less than $\eta'$ from $\partial D$.
\end{corollary}
\begin{proof}
	Fix $\epsilon>0$. By Theorem \ref{theorem schramm}, there exists $j_0(\epsilon)$ such that the following event, denoted by $C'(j_0(\epsilon))$, has probability at most $\epsilon/2$: there exists at least one loop erased after sampling $j_0(\epsilon)$ branches that has diameter at least $\epsilon$.
	By Lemma \ref{beurling estimate}, there exist $\eta >0$ and $\delta(\eta)>0$ such that for any $\delta <\delta(\eta)$ and any vertex $v\in \Gamma\dis \cap (\partial D)^{\eta}$, 
	\[ \mathbb{P}_v (X^{\#\delta} \textit{ exits } B(v,\epsilon /2) \textit{ without hitting } \partial \Gamma^{\#\delta}_D ) < \frac{\epsilon}{2j_0(\epsilon)}. \]
	For any $j\in \mathbb{N}^*$, let $\mathcal{T}_j\dis$ be the union of the first $j$ branches. Let $\tau ^{\#\delta,j}$ be the hitting time of $(\partial D)^{\eta}\cup \mathcal{T}_{j-1}\dis$ by the $j$-th random walk in Wilson's algorithm. For all $j\in \mathbb{N}^*$, let $C(j,\eta)$ be the event that the $j$-th random walk, after time $\tau^{\#\delta,j}$, exits $ B(X\dis (\tau ^{\#\delta,j}),\epsilon /2) $ without hitting $\partial \Gamma^{\#\delta}_D \cup \mathcal{T}_{j-1}\dis$. Then the probability of $C(j,\eta)$ is bounded by $\frac{\epsilon}{2j_0(\epsilon)}$, uniformly over $j$.
	
	Denote by $A(\epsilon,\eta)$ the set of loops in D that have diameter larger than $\epsilon $ and distance less than $\eta$ from $\partial D$. By a union bound, for any $\delta < \min (6^{-j_0(\epsilon)}\delta_0,\delta(\eta))$, we have
	\begin{equation*}
		\mathbb{P}(\mathcal{L}\dis_D \cap A(\epsilon,\eta) \not=\emptyset ) \le \mathbb{P}\left( (\cup_{j\le j_0(\epsilon)} C(j,\eta) ) \cup C'(j_0(\epsilon)) \right)
		\le \sum_{j\le j_0(\epsilon)} \mathbb{P} ( C(j,\eta) )  +\epsilon/2
		\le \epsilon.
	\end{equation*}
	To bound the loop measure of $A(\epsilon,\eta')$, since the random walk loop soup is a Poisson point process, it suffices to take $\eta '$ such that $\mathbb{P}(\mathcal{L}\dis_D \cap A(\epsilon,\eta') \not=\emptyset ) \le 1- \exp(-\epsilon)$.
\end{proof}

\subsection{Concatenation Probability}\label{Concatenation probability}
By Corollary \ref{loop addtion2}, under the good ordering, for $\epsilon>0$, we can couple $(\mathcal{L}^{\#\delta} , \mathbb{L}^{\#\delta}, \mathbb{L}^{\#\delta} _{\epsilon})$, such that each loop in $\mathbb{L}^{\#\delta}_{j} $ is a concatenation of loops in $\mathcal{L}^{\#\delta}$ with a chosen root. 
More precisely, for each vertex $v$, let $l^{\#\delta}_v \in \mathbb{L}^{\#\delta}$ be the loop erased in Wilson's algorithm. Under this coupling, there are $M^{\#\delta}(v)$ loops rooted at $v$,  $l_{v,1}^{\#\delta} , \cdots , l^{\#\delta}_{v,M(v)} \in \mathcal{L}^{\#\delta}$, such that
\[ l_v^{\#\delta} = l_{v,1}^{\#\delta} \oplus \cdots \oplus l^{\#\delta}_{v,M^{\#\delta}(v)}.\]
We define the number of $\epsilon$-large loops rooted at $v$ by $M^{\#\delta}_{\epsilon}(v)=\# \{ m\le M(v): \diameter(l\dis_{v,m}) >\epsilon \}$.
Then we prove that, with high probability, the loops erased in some branches of Wilson's algorithm are concatenated by one large loop and some tiny loops in RWLS. 
\begin{lemma}\label{proba of coming back}
	For $j>0$, let $\mathcal{T}_j\dis$ denote the union of first $j$ branches in Wilson's algorithm, and let $Y^{j+1,\#\delta}$ denote the $(j+1)$-th branch. Let $N^{j+1,\#\delta}$ be the number of iterations in the greedy Wilson's algorithm. Then for $\epsilon>0$, for any $K>0$, as $\delta \to 0$, 
	\[\mathbb{P}( N^{j+1}\le K;\ \exists v\in Y^{j+1,\#\delta} : M\dis _{4\epsilon }( v ) \ge 2) |\mathcal{T}_j\dis )=K o_{\delta} (1).\]
\end{lemma}
\begin{proof}
	Let $(X^{j+1,\#\delta}(t):0\le t\le \tau )$ be the random walk coupled with $Y^{j+1,\#\delta}$, where $\tau$ is the time when hitting $\partial D \cup \mathcal{T}_j\dis$. 
	For $v\in Y^{j+1,\#\delta}$, if  $M^{\#\delta}_{4\epsilon} (v) \ge 2$, i.e. at least two loops with a diameter of more than $4\epsilon$ are rooted at $v$, then in the greedy Wilson's algorithm, $\# \{n\in [0,N^{j+1,\#\delta}]:Y^{j+1,\#\delta}_n(s^{j+1,\#\delta}_{n+1}) =v \}\ge 2$, and the random walk between two of these points has a diameter of more than $4\epsilon$. In other words, there exist indices $n_1< n_2 \le N$ and the corresponding times $t_1, t_2 \in[0,\tau]$ in the random walk, such that $v=Y^{j+1,\#\delta}_{n_1} (s^{j+1,\#\delta}_{n_1+1})=Y^{j+1,\#\delta}_{n_2}(s^{j+1,\#\delta}_{n_2+1})=X^{j+1,\#\delta}(t_1)=X^{j+1,\#\delta}(t_2)$, and $\diameter (X^{j+1,\#\delta}[t_1,t_2]) \ge 4\epsilon$. 
	
	For $n\in [0,N^{j+1}-1]$, let $\tau'_n = \inf \{t>\tau^{j+1}_{n+1}: X^{j+1}(t) \in (B(v,\epsilon))^c \}$, let $\tau''_n = \inf \{t>\tau '_n: X^{j+1}(t)=v  \}$.
	We say that $Y^{j+1,\#\delta}_n$ is $\epsilon$-returnable if the random walk will, before hitting the boundary $\partial D \cup \mathcal{T}_j$, run a distance of $\epsilon$ from $v$ and return to $v$, precisely $\tau'_n< \tau_n''<\tau_{\partial D \cup \mathcal{T}_j}$.
	Then for $v\in Y^{j+1,\#\delta}$, if  $M^{\#\delta}_{4\epsilon} (v) \ge 2$, we have $v=Y^{j+1,\#\delta}_{n} (s^{j+1,\#\delta}_{n+1})$ for some $n$ and that $Y^{j+1,\#\delta}_n$ is $\epsilon$-returnable.
	
	Next we show that the probability of a point being $\epsilon$-returnable is uniformly $o_{\delta}(1)$ as $\delta \to 0$, uniformly in the choice of the points and the sampled trees. For $\epsilon'>0$, consider a 2D Brownian motion starting from $\epsilon$, let $t_{\epsilon' }$ be its hitting time of $B(0,\epsilon')$, let $t_D$ be its hitting time of $\partial D$.  By stopping time theorem and the fact that $\mathbb{E} [ \log(\| B_{t_{\epsilon'} \wedge t_D}\|)|t_{\epsilon'} > t_D]  \le \log(\diameter(D))\mathbb{P}(t_{\epsilon'} > t_D)$, we have
	\begin{equation}\label{2D BM estimate}
		\mathbb{P}(t_{\epsilon'} < t_D) \le \frac{\log(\diameter(D)) -\log \epsilon}{\log(\diameter(D))-\log \epsilon'}.   
	\end{equation}
	By the invariance principle, for each vertex $w \in \partial^{\#\delta } B(v,\epsilon)$,
	\[ \mathbb{P}^{\#\delta}(\text{The random walk emanating from $w$ hits $v$ before exitting }D ) =o_{\delta}(1). \]
	Then the probability of some point being $\epsilon$-returnable is $o_{\delta}(1)$. It follows that
	\begin{align*}
		&\mathbb{P}( N^{j+1}\le K;\ \exists v\in Y^{j+1,\#\delta} : M\dis _{4\epsilon }( v ) \ge 2) |\mathcal{T}_j\dis )\\
		&\le \mathbb{P}(N^{j+1}\le K; \exists n\in [0,N^{j+1}], Y^{j+1,\#\delta}_{n}(s_{n+1}) \text{ is $\epsilon$-returnable} )\\
		&\le \sum_{i=1}^K\mathbb{P}(Y^{j+1,\#\delta}_{n}(s_{n+1}) \text{ is $\epsilon$-returnable})\\
		&=Ko_{\delta}(1).
	\end{align*}
\end{proof}

An similar estimate is the ERROR term:
\begin{lemma}\label{error}
	For every $\epsilon >0$, there exists $\delta_e>0$ such that for all $\delta>\delta_e$, the probability of 'ERROR' occurring when sampling some branch is uniformly bounded by $\frac{2\log(\diameter(D)) -2\log \epsilon}{\log(\diameter(D))-\log r(\epsilon)}$.
\end{lemma}
\begin{proof}
	The result follows by taking $\epsilon'=r(\epsilon)$ in Equation \eqref{2D BM estimate} and the Invariance principle.
\end{proof}

\subsection{Proof of Theorem \ref{scaling limit}}
Now we give a proof for Theorem \ref{scaling limit}.
\begin{proof}
	We first prove that the sequence $(\mathcal{L}^{\#\delta}_D)_{\delta>0}$ is Cauchy in the Lévy-Prokhorov metric associated with the topology on $\mathcal{M}(D)$. Suppose first that $\mathbb{C}\backslash D$ is locally connected.
	
	Given $\epsilon>0$, by Theorem \ref{theorem schramm}(Schramm's finiteness theorem), there exist $j_0= j_0(\epsilon)\in \mathbb{N}$ and $\delta_{\epsilon}>0$, such that for all $\delta < \delta_{\epsilon}$, 
	\begin{equation}\label{eq1}
		d_{LP} (\mathcal{L}^{\#\delta }_D, \mathcal{L}^{\#\delta}_{ j_0}) \le \epsilon.
	\end{equation}
	We focus on these $j_0$ branches. For each $j>0$, let $\mathcal{E}^{\#\delta}_j$ be the event that ``ERROR'' occurs in the $j-$th branch in the greedy Wilson's algorithm. Take
	\begin{equation*}
		r(\epsilon)=\exp\big(-2\epsilon ^{-1}j_0(\epsilon )(\log(\diameter (D))-\log \epsilon)+\log(\diameter (D))\big).
	\end{equation*}
	By lemma \ref{exp tail}, lemma \ref{proba of coming back} and lemma \ref{error}, we have:	
	\begin{align*}
		\mathbb{P} (d(\mathbb{L}^{\#\delta}_{\epsilon ,j_0}, \mathcal{L}^{\#\delta} _{j_0} )> 9\epsilon)
		\le & \mathbb{P}(\bigcup_{ j\le j_0} (\mathcal{E}^{\#\delta}_j \cup \{ \exists s\le S^j, M\dis_{4\epsilon }( Y_n^{j,\#\delta}(s) ) \ge 2\}) )\\
		\le& \sum_{j=1}^{j_0} \mathbb{P}(\mathcal{E}^{\#\delta}_j)+\sum_{j=1}^{j_0} \mathbb{P}( \exists s\le S^j, M_{4\epsilon }\dis( Y_n^{j,\#\delta}(s) ) \ge 2)\\
		\le& j_0(\epsilon)\frac{2\log(\diameter(D)) -2\log \epsilon}{\log(\diameter(D))-\log r(\epsilon)} \\
		&+\sum_{j=1}^{j_0} \mathbb{P}(N_j \ge K) +  \sum_{j=1}^{j_0}\mathbb{P}( N^{j}\le K;\ \exists v\in Y^{j,\#\delta} : M\dis _{4\epsilon }( v ) \ge 2)\\
		\le & \epsilon + j_0 (\epsilon )\beta(\epsilon) \exp(-\alpha(\epsilon) K) + j_0(\epsilon) K o_{\delta}(1)\\
		\le & 2\epsilon +j_0(\epsilon) K o_{\delta}(1).
	\end{align*} 
	Take $K= \lceil \alpha(\epsilon)^{-1} \log(\epsilon^{-1} j_0(\epsilon )\beta(\epsilon)) \rceil.$
	Take $\delta' >0$ such that for all $ \delta <\delta'$, the last term of the above equation is smaller than $\epsilon$, then  $\mathbb{P} (d(\mathbb{L}^{\#\delta}_{\epsilon, j_0}, \mathcal{L}^{\#\delta} _{j_0} )> 9\epsilon) \le 3\epsilon.$ Moreover, for all $ \delta <\delta'$, we have
	\begin{equation}\label{eq2}
		d_{LP}(\mathbb{L}^{\#\delta}_{\epsilon, j_0}, \mathcal{L}^{\#\delta} _{j_0} ) \le 9\epsilon.
	\end{equation}
	By corollary \ref{k branches}(with a remark that in the continuum version we choose the center point of each cell as the starting point, then the hypothesis is naturally satisfied), there exists $\delta'' >0$, such that for all $\delta _1, \delta_2< \delta''$,
	\begin{equation}\label{eq3}
		d_{LP} ( \mathbb{L}^{\#\delta_1}_{\epsilon,j_0}, \mathbb{L}^{\#\delta_2}_{\epsilon, j_0} ) \le \epsilon. 
	\end{equation}
	Combining the Equations \eqref{eq1}, \eqref{eq2} and \eqref{eq3}, for all $\delta_1, \delta_2 <\min \{\delta_{\epsilon} ,\delta', \delta''\}$, we have
	\begin{align*}
		d_{LP}(\mathcal{L}^{\#\delta_1}_D, \mathcal{L}^{\#\delta_2}_D)&\le 
		d_{LP} (\mathcal{L}^{\#\delta_1 }_D, \mathcal{L}^{\#\delta_1}_{ j_0})+
		d_{LP} (\mathcal{L}^{\#\delta_1}_{ j_0},\mathbb{L}^{\#\delta_1 }_{\epsilon, j_0} )
		+d_{LP} ( \mathbb{L}^{\#\delta_1}_{\epsilon,j_0}, \mathbb{L}^{\#\delta_2}_{\epsilon, j_0} )\\
		&\ \ \ \ +d_{LP} (\mathbb{L}^{\#\delta_2 }_{\epsilon, j_0}, \mathcal{L}^{\#\delta_2}_{ j_0})+
		d_{LP} ( \mathcal{L}^{\#\delta_2}_{ j_0},\mathcal{L}^{\#\delta_2 }_D)\\
		&\le 21 \epsilon.
	\end{align*}
	Then the sequence $(\mathcal{L}^{\#\delta}_D)_{\delta>0}$ is Cauchy in the case where $\mathbb{C}\backslash D$ is locally connected. If $\mathbb{C}\backslash D$ is not locally connected, by Corollary \ref{boundary}, for all $\epsilon>0$, there exist $\eta>0$ and a domain $D'\subset D$, such that $\partial D'\subset (\partial D)^{\eta}$, $\mathbb{C}\backslash D'$ is locally connected and $d_{LP}(\mathcal{L}\dis_D,\mathcal{L}\dis_{D'})\le \epsilon$. Since $(\mathcal{L}^{\#\delta}_{D'})_{\delta>0}$ is Cauchy, we conclude that $(\mathcal{L}^{\#\delta}_D)_{\delta>0}$ is also Cauchy.
	
	According to \cite{lawler2004randomwalkloopsoup} (with more details provided in Appendix \ref{grid lattice}), the limit for the sequence of grid lattice is $\mathcal{L}_D$. Since the above proof holds for all sequences of graphs satisfying the hypothesis, we can interlace two such sequences and obtain the same limit. Hence the limit must be $\mathcal{L}_D$.
\end{proof}

\subsection{Proof of Corollary \ref{cor}}
This proof only uses the fact that the loop soups are Poisson point processes and that these geometrical properties that we consider are closed for the topology of loop space.
\begin{proof}
	Take $\epsilon_0>0$ such that $A\subset \{l\in \mathscr{L}_D; \diameter(l) \ge \epsilon_0\}$. 
	We define the counting function \[N(\cdot,A): \mathcal{M}(D)\to \mathbb{N};\ N((l_i)_{i\in \mathbb{N}}, A) := \card \{ i\in \mathbb{N}: l_i \in A\}. \]
	We prove that $N(\cdot, A)$ is continuous conditioned that there are no loops in $\partial A$. Assume that $L^n=(l_i^n)_{i\in \mathbb{N}}$ converges to $L=(l_i)_{i\in \mathbb{N}}$ in the space of loop soups. Assume that no loop in $L$ is contained in $\partial A$. Since $A$ is an open set, we take $\epsilon_1 <\epsilon_0/2$ such that $\cup_{i\in \mathbb{N}:l_i\in A} B_\mathscr{L}(l_i ,\epsilon_1) \subset A$ and $\cup_{i\in \mathbb{N}:l_j\notin A,\diameter{l}\ge \epsilon_0 /2} B_\mathscr{L}(l_j ,\epsilon_1) \subset A^c$. If $d_{\mathcal{M}(D)}(L^n, L)< \frac{\epsilon_1}{4}$, then there exists $\lambda_n: I\to \mathbb{N}$ such that 
	\[ \sup\{ d(l_i,l_{\lambda_n(i)}^n), \forall i\in I; \frac{\diameter(l_i)}{2}, \forall i\in \N \backslash I; \frac{\diameter (l^n_j)}{2}, \forall j\in \N \backslash \lambda_n (I)\}< \frac{\epsilon_1}{2}. \]
	Then every loop $l^n_j\in L^n$  either has a diameter smaller than $\epsilon_1$ or is paired with some loops $l_i=l_{\lambda^{-1}_n(j)}\in L$ with $d(l_i,l_{\lambda_n(i)}^n)<\epsilon_1/2$. For the unpaired loop, two cases may occur. If $l_i\in A$, then $l^n_j\in A$. Otherwise, we further distinguish two subcases: if $\diameter(l_i)< \epsilon_0/2$, then $\diameter (l^n_j)<\epsilon_0/2 +\epsilon_1< \epsilon_0 $; while if $\diameter (l_i)>\epsilon_0/2$, then $l^n_j\in B(l_i,\epsilon_1)\subset A^c$. Hence $N(L^n,A)= N(L,A)$, which proves the continuity.
	
	Since the Brownian loop soup almost surely contains no loop in $\partial A$, the map $N(\cdot,A) $ is almost surely continuous, the result follows from Theorem \ref{scaling limit}.
\end{proof}

\begin{appendix}
	\section{Particular case: grid lattice}\label{grid lattice}
	In this section, we state the result of Lawler and Trujillo Ferreras \cite{lawler2004randomwalkloopsoup} and deduce Theorem \ref{scaling limit} in the case $\Gamma^{\#\delta}=\delta \mathbb{Z}^2\cap D$.
	We first introduce some notation to state the theorem. For each positive integer $N$, we define the scaling function on the set of loops by 
	\[ \Phi_N \gamma(t) :=N^{-1} \gamma(tN^2), \ 0\le t\le \frac{t_{\gamma}}{N^2}. \]
	Due to the discretization, we introduce the following auxiliary functions. For $t\ge (5/8) N^{-2}$ and positive integer $k$, we let 
	\[ \varphi_N(t):=\frac{k}{N^2} \ \text{ if }\ \frac{k}{N^2}-\frac{3}{8N^2}\le t< \frac{k}{N^2} +\frac{5}{8N^2}. \]
	To distinguish the root, for $z\in \mathbb{C}$ and $z_0\in \mathbb{Z}^2$, we define \[ \psi _N(z) :=\frac{z_0}{N}\  \text{ if }\ \max\{ |\Re(Nz-z_0)|, |\Im (Nz-z_0)|\} <\frac{1}{2}. \]
	
	For the random walk loop soup, for each positive integer $N$, we regard the loop $\gamma^{\#\frac{1}{N}}$ as a loop scaled from the random walk loop over $\mathbb{Z}^2$, i.e., there is a loop over $\mathbb{Z}^2$, $\gamma ^{rw}$, such that
	\[ \gamma^{\#\frac{1}{N}} =\frac{1}{N} \gamma^{rw} (2N^2 t), \ \ 0\le t\le \frac{t_{\gamma^{rw}}}{2N^2}. \]
	
	\begin{theorem}[Corollary 5.4 \cite{lawler2004randomwalkloopsoup}]\label{lawler}
		There exists $c>0$ such that for every $N\in \mathbb{N}$ and $\theta<2$; there exists a coupling of a Brownian loop soup restricted to $D$ and a $(1/N)$-random walk soup restricted to $D$ such that, except on an event of probability at most $c (\log N)^{1/2} N^{2-(5/4)\theta}$, we have the one-to-one correspondence between 
		\[ \{ \gamma \in \Phi_N (\mathcal{L}): \varphi_N(t_{\gamma})>N^{\theta-2},\ |\psi_N(\gamma(0))|<1\} \]
		and \[ \{ \gamma^{\#\frac{1}{N}} \in \mathcal{L}^{\#\frac{1}{N}} : t_{\gamma^{\# 1/N }} >N^{\theta-2}, |\gamma^{\#\frac{1}{N}}(0)|<1\}. \]
		If $\gamma$ and $\gamma^{\#1/N}$ are paired in this correspondence, then 
		\[|t_{\gamma}-t_{\gamma^{\#1/N}}| \le \frac{5}{8} N^{-2}, \]
		\[ \sup_{s\in [0,1]} |\gamma (st_{\gamma})-\gamma^{\#\frac{1}{N}} (st_{\gamma^{\#1/N}})|\le c \frac{\log N}{N}. \]
	\end{theorem}
	Next we prove the main result for the grid lattice case.
	\begin{corollary}
		Let $\Gamma^{\#\frac{1}{N}}=\frac{1}{N}\mathbb{Z}^2 \cap D$. In the sense of the loop soup topology, as $N\to +\infty$, the random walk loop soup converges in law to the Brownian loop soup, i.e.
		\[ \mathcal{L}_D^{\#\frac{1}{N}\mathbb{Z}^2} \xrightarrow[N\to +\infty]{(d)} \mathcal{L}_D. \]
	\end{corollary}
	\begin{proof}
		Without loss of generality, assume that $D$ is contained in the unit disk $\mathbb{D}$.
		For $\epsilon>0$, let $E(\epsilon)$ be the following event,
		\[ E(\epsilon) =\{\# \{ \gamma \in \Phi_N (\mathcal{L}): \varphi_N (t_{\gamma})\le N^{\theta-2}, |\psi _N(\gamma(0))|<1, \gamma\in D,\  \diameter (\gamma) \ge \epsilon\} \ge 1\}.\]
		Since the Brownian loop soup is a Poisson point process, we have
		\begin{align*}
			\mathbb{P} (E(\epsilon)) &= 1-\exp (- \mu(\varphi_N (t_{\gamma})\le N^{\theta-2}, |\psi _N(\gamma(0))|<1, \gamma\in D,\  \diameter (\gamma) \ge \epsilon))\\
			&\le 1-\exp (-2\int_D \int _0^{N^{\theta-2}} \frac{1}{2\pi t^2} \mu^{br}_t(z) (\diameter (\gamma  )\ge \epsilon ) dt\, dz )\\
			&\le 2\int_D \int _0^{N^{\theta-2}} \frac{1}{2\pi t^2} \mu^{br}_t(z) (\diameter (\gamma  )\ge \epsilon ) dt\, dz\\
			&\le 2area(D) \int_0^{N^{\theta-2}} \frac{1}{2\pi t^2} \mu^{br} (\diameter (\gamma)\ge \epsilon t^{-\frac{1}{2}}) \, dt.
		\end{align*}
		Since we can write a standard 2D Brownian bridge $W_t=B_t-tB_1 =W_t^{x} +i W_t^y$ where $W_t^x$ and $W_t^y$ are two 1D Brownian bridge, and by the property of 1D Brownian bridge, we have
		\begin{align*}
			\mu^{br} (\diameter (\gamma)\ge \epsilon t^{-\frac{1}{2}}) &\le 2\mathbb{P} (\sup_{s\in[0,1]} |W_s^x| \ge \frac{\epsilon}{2\sqrt{2t}})\\
			&\le 4 \mathbb{P}(\sup_{s\in[0,1]} W^x_s \ge \frac{\epsilon}{ 2\sqrt{2t}})\\
			&=4\exp (-\frac{\epsilon^2}{4t}).
		\end{align*}
		So by calculation, we have
		\begin{equation*}
			\mathbb{P} (E(\epsilon)) \le area(D) \frac{16}{\pi \epsilon^2} \exp(-\frac{\epsilon^2}{4}N^{2\theta}).
		\end{equation*}
		
		Similarly we denote by $E^{\#\frac{1}{N}}(\epsilon)$ the event
		\[ E^{\#\frac{1}{N}}(\epsilon) =\{\# \{ \gamma^{\#\frac{1}{N}} \in \mathcal{L}^{\# \frac{1}{N}} : t_{\gamma^{\# \frac{1}{N} }} \le N^{\theta-2}, |\gamma^{\#\frac{1}{N}}(0)|<1,\gamma^{\#\frac{1}{N}} \subset D, \diameter (\gamma^{\frac{1}{N}}) \ge \epsilon\} \ge 1\}. \]
		By the definition of random walk loop soup, we have
		\begin{align*}
			\mathbb{P}(E^{\#\frac{1}{N}}) &= 1- \exp( -\Lambda^{\#\delta } (t_{\gamma^{\#1/N}} \le N^{\theta-2}, |\gamma^{\#\frac{1}{N}}(0)|<1, \gamma^{\#\frac{1}{N}} \subset D, \diameter (\gamma^{\frac{1}{N}}) \ge \epsilon ) )\\
			&\le \Lambda^{\#\delta } (t_{\gamma^{\#1/N}} \le N^{\theta-2}, |\gamma^{\#\frac{1}{N}}(0)|<1,\gamma^{\#\frac{1}{N}} \subset D,\diameter (\gamma^{\frac{1}{N}}) \ge \epsilon)\\
			&\le \sum_{z\in D^{\#1/N}}\sum_{n=\frac{N\epsilon}{2\sqrt{2}}}^{N^{\theta}} \frac{1}{2n} 4^{-2n} \#\{ \gamma^{\#\frac{1}{N}}: | \gamma^{\#\frac{1}{N}}|=2n,\diameter (\gamma^{\#\frac{1}{N}}) \ge \epsilon \}.
		\end{align*}
		We can consider a simple random walk over $\mathbb{Z}^2$ denoted by $S^{2D}_n$, there are two independent 1D simple random walks $S^x$ and $S^y$ such that $S_j^{2D}=\frac{S_j^x+i S_j^y}{1+i}$, then 
		\begin{align*}
			4^{-2n} \#\{ \gamma^{\#\frac{1}{N}}: | \gamma^{\#\frac{1}{N}}|=2n,\diameter (\gamma^{\#\frac{1}{N}}) \ge \epsilon \} & = \mathbb{P}(S^{2D}_{2n}=0, \diameter (S^{2D}) \ge N\epsilon)\\
			&\le 4\mathbb{P} (S^x_{2n}=0, \sup_{0\le j\le 2n} (S^x_j) \ge \frac{N\epsilon}{2\sqrt{2}} )\\
			&\le 4\mathbb{P} ( \sup_{0\le j\le 2n} (S^x_j) \ge \frac{N\epsilon}{2\sqrt{2}} ).
		\end{align*}
		By Hoeffding's inequality, we have 
		\[   \mathbb{P} ( \sup_{0\le j\le 2n} (S^x_j) \ge \frac{N\epsilon}{2\sqrt{2}} ) \le 
		2n \exp(-\frac{N^2\epsilon^2}{32n}).\]
		Then 
		\begin{align*}
			\mathbb{P}(E^{\#\frac{1}{N}}(\epsilon)) &\le \sum_{z\in D^{\# 1/N}} \sum_{n=\frac{N\epsilon}{2\sqrt{2}}}^{N^{\theta}} 4\exp(-\frac{N^2\epsilon^2}{32n})\\
			&\le 4area(D) N^{2+\theta} \exp(-\frac{\epsilon^2}{32n} N^{2-\theta}). 
		\end{align*}
		
		By Theorem \ref{lawler}, take $\theta=\frac{9}{5}$, for $N$ large enough such that \[  c (\log N)^{1/2} N^{2-(5/4)\theta}+area(D) \frac{16}{\pi \epsilon^2} \exp(-\frac{\epsilon^2}{4}N^{2\theta}) + 4area(D) N^{2+\theta} \exp(-\frac{\epsilon^2}{32n} N^{2-\theta})<\epsilon,\]
		and $c\log N/N <\epsilon$,
		we have
		\begin{align*}
			\mathbb{P}(d_{\mathcal{M}(D)} (\mathcal{L}_D^{\#\frac{1}{N}},\mathcal{L}_D) \ge \epsilon) &\le c (\log N)^{1/2} N^{2-(5/4)\theta} + \mathbb{P}( E(\epsilon) \cup E^{\#\frac{1}{N}}(\epsilon) )\\
			&\le  c (\log N)^{1/2} N^{2-(5/4)\theta} + \mathbb{P}( E(\epsilon)) +\mathbb{P}( E^{\#\frac{1}{N}}(\epsilon) )\\
			&<\epsilon.
		\end{align*}
		Then we conclude by using the Lévy-Prokhorov distance.
	\end{proof}
	
	\section{Proof of Theorem \ref{theorem schramm}}\label{appendix B}
	Before giving the proof of Theorem \ref{theorem schramm}, we state an essential lemma which is a type of Beurling estimate. Compared with lemma \ref{beurling estimate}, this lemma gives a quantitative estimate and depends only on the uniform crossing estimate.
	\begin{lemma}[Lemma 4.17 in \cite{berestycki2018dimersimaginarygeometry}] \label{beurling}
		There exist constants $C,c,c'>0$ such that for all $connected$ set $K \subset \mathbb{C}$ such that the diameter of $K$ is at least $R$, let $v$ be an vertex, let $\delta_0$ be the same constant as the uniform crossing assumption,  then for all $\delta\in (0, C\  dist(v,K) \delta_0)$,
		\[ \mathbb{P}_v ( X \text{ exits } B(v,R)^{\#\delta} \text{ before hitting } K^{\#\delta}) \le c(\frac{dist (v,K)}{R})^{c'}. \]
	\end{lemma}
	
	\begin{proof}[Proof of Theorem \ref{theorem schramm}]
		Let $Q_1^{\#\delta
		},Q_2^{\#\delta},\cdots$ be the ordering of starting points defined in the good ordering. We denote the union of the branches sampled from the vertices $Q^{\#\delta}_1, Q^{\#\delta}_2,\cdots ,Q^{\#\delta}_j$ in Wilson's algorithm by $T_j^{\#\delta}$. For $j>2$, the number of vertices in $Q_j$ is at most $d 6^j$ where $d$ is a universal constant. Notice that each vertex in $D^{\#\delta}$ is within distance $4\cdot 6^{-j}$ from a vertex in $Q_{j-1}^{\#\delta}$.  Let $j_{max}:=\lfloor \log_6 (\frac{C\delta_0 }{\delta}) \rfloor$. For $j\in [3, j_{max}] \cap \mathbb{N}$, for each vertex $w\in Q_j^{\#\delta}$, $dist(w,T_j^{\#\delta})\le 4\cdot 6^{-j}$. If we ignore the nearest vertex chosen in the cell in $Q_{j-1}^{\#\delta}$, this distance is at least $6^{-j}$, so the hypothesis in lemma \ref{beurling} is satisfied. In other words, for $j\le j_{max}$, for all $\delta < C 6^{-j} \delta_0$ and all $v\in Q_j$,
		\[ \mathbb{P}_v ( X \text{ exits } B(v,C_0 6^{-j})^{\#\delta} \text{ before hitting } T_{j-1}^{\#\delta}) \le c(\frac{dist (v,T_{j-1}^{\#\delta})}{C_0 6^{-j}})^{c'}\le c(\frac{4\cdot 6^{-j}}{C_0 6^{-j}})^{c'} , \]
		where $C_0$ is a universal constant such that $c(4/C_0)^{c'} <1/2$.
		For $v\in Q_j^{\#\delta}$, let $D(v,j)$ be the event that the random walk emanating from $v\in Q_j^{\#\delta}$ reaches distance $j^2 6^{-j}$ without hitting $T_{j-1}^{\#\delta }$. By using the Markov property, we can iteratively apply the above estimate for the walk $j^2/C_0$ times, then
		\[ \mathbb{P}(D(v,j))\le (\frac{1}{2})^{j^2/C_0}. \]
		
		When $j>j_{max}$, for $v\in Q_j^{\#\delta}$, let $D(v,j)$ be the event that the random walk emanating from $v\in Q_j^{\#\delta}$ reaches distance $j_{max}^2 6^{-j_{max}}$ without hitting $T_{j-1}^{\#\delta}$ (and in particular $T_{j_{\max}}^{\#\delta}$). Then we still have $ \mathbb{P}(D(v,j))\le (\frac{1}{2})^{j_{max}^2/C_0}$. With the bounded density assumption in Section \ref{graph}, the number of vertices in $\cup_{j>j_{max}} Q_j^{\#\delta}$ is at most $B 6^{j_{max}}$ for some constant $B$ depending only on $\delta_0$ and the constant $C_b$ in the bounded density assumption (in fact we can take $B=2\frac{C_b}{C \delta_0}$). 
		
		Now for an integer $j'$, we define an event 
		\[ D(j') :=\bigcup_{j\ge j'} \bigcup _{v\in Q_j^{\#\delta}} D(v,j). \]
		By a union bound, we have 
		\begin{align*}
			\mathbb{P}(D(j')) &\le \sum_{j'\le j\le j_{max}} d6^j (1/2)^{j^2 /C_0} + B6^{j_{max}} (1/2)^{j_{max}^2/C_0}\\
			&\le \sum_{j\ge j'} (B+d)6^j (1/2)^{j^2/C_0}.
		\end{align*}
		We take $j'=j'(\epsilon)$ such that $\sum_{j\ge j'} (B+d)6^j (1/2)^{j^2/C_0}\le \epsilon$ and $j'^26^{-j'} \le \epsilon/2$. Then with probability at least $1-\epsilon$, no random walk emanating from a vertex in $Q_j^{\#\delta}$ with $j\ge j'$ reaches distance $\epsilon/2$ without hitting $T_{j-1}^{\#\delta}$. We conclude by taking $j_0(\epsilon)=d 6^{j'+1}>\sum_{j<j'}\#Q_j^{\#\delta}$. The rest of the theorem follows immediately by considering the coupling described in Corollary \ref{loop addtion2}.
	\end{proof}
\end{appendix}

\subsection*{Acknowledgements}
The author is very grateful to Benoît Laslier for suggesting this question, for his numerous inspiring discussions, and for his help during the preparation of the manuscript. This work was supported by China Scholarship Council (CSC) program (File No. 202306340027).

	{\small \bibliographystyle{plain}
		\bibliography{references}}

\begin{thebibliography}{10}

\bibitem{berestycki2018dimersimaginarygeometry}
Nathana\"el Berestycki, Beno\^it Laslier, and Gourab Ray.
\newblock Dimers and imaginary geometry.
\newblock {\em Ann. Probab.}, 48(1):1--52, 2020.

\bibitem{berestycki2024dimersriemannsurfacesii}
Nathana\"el Berestycki, Beno\^it Laslier, and Gourab Ray.
\newblock Dimers on {R}iemann surfaces {II}: {C}onformal invariance and scaling limit.
\newblock {\em Probab. Math. Phys.}, 5(4):961--1037, 2024.

\bibitem{berestycki2024dimersriemannsurfacesi}
Nathana\"el Berestycki, Beno\^it Laslier, and Gourab Ray.
\newblock Dimers on {R}iemann surfaces {I}: {T}emperleyan forests.
\newblock {\em Ann. Inst. Henri Poincar\'e{} D}, 12(2):363--444, 2025.

\bibitem{chang2014markovloopsdiscretespaces}
Yinshan Chang and Yves Le~Jan.
\newblock Markov loops in discrete spaces.
\newblock In {\em Probability and statistical physics in {S}t. {P}etersburg}, volume~91 of {\em Proc. Sympos. Pure Math.}, pages 215--271. Amer. Math. Soc., Providence, RI, 2016.

\bibitem{kemppainen2015randomcurvesscalinglimits}
Antti Kemppainen and Stanislav Smirnov.
\newblock Random curves, scaling limits and {L}oewner evolutions.
\newblock {\em Ann. Probab.}, 45(2):698--779, 2017.

\bibitem{Lawler_Limic_2010}
Gregory~F. Lawler and Vlada Limic.
\newblock {\em Random walk: a modern introduction}, volume 123 of {\em Cambridge Studies in Advanced Mathematics}.
\newblock Cambridge University Press, Cambridge, 2010.

\bibitem{lawler2003conformalinvarianceplanarlooperased}
Gregory~F. Lawler, Oded Schramm, and Wendelin Werner.
\newblock Conformal invariance of planar loop-erased random walks and uniform spanning trees.
\newblock {\em Ann. Probab.}, 32(1B):939--995, 2004.

\bibitem{lawler2004randomwalkloopsoup}
Gregory~F. Lawler and Jos\'e{}~A. Trujillo~Ferreras.
\newblock Random walk loop soup.
\newblock {\em Trans. Amer. Math. Soc.}, 359(2):767--787, 2007.

\bibitem{lawler2003brownianloopsoup}
Gregory~F. Lawler and Wendelin Werner.
\newblock The {B}rownian loop soup.
\newblock {\em Probab. Theory Related Fields}, 128(4):565--588, 2004.

\bibitem{jan2008markovloopsrenormalization}
Yves Le~Jan.
\newblock Markov loops and renormalization.
\newblock {\em Ann. Probab.}, 38(3):1280--1319, 2010.

\bibitem{jan2010markovpathsloopsfields}
Yves Le~Jan.
\newblock {\em Markov paths, loops and fields}, volume 2026 of {\em Lecture Notes in Mathematics}.
\newblock Springer, Heidelberg, 2011.
\newblock Lectures from the 38th Probability Summer School held in Saint-Flour, 2008, \'Ecole d'\'Et\'e{} de Probabilit\'es de Saint-Flour.

\bibitem{pommerenke2013boundary}
Ch. Pommerenke.
\newblock {\em Boundary behaviour of conformal maps}, volume 299 of {\em Grundlehren der mathematischen Wissenschaften [Fundamental Principles of Mathematical Sciences]}.
\newblock Springer-Verlag, Berlin, 1992.

\bibitem{qian2026couplingbrownianloopsoups}
Wei Qian.
\newblock Coupling brownian loop soups and random walk loop soups at all polynomial scales, 2026.

\bibitem{MR3877544}
Artem Sapozhnikov and Daisuke Shiraishi.
\newblock On {B}rownian motion, simple paths, and loops.
\newblock {\em Probab. Theory Related Fields}, 172(3-4):615--662, 2018.

\bibitem{schramm1999scalinglimitslooperasedrandom}
Oded Schramm.
\newblock Scaling limits of loop-erased random walks and uniform spanning trees.
\newblock {\em Israel J. Math.}, 118:221--288, 2000.

\bibitem{sheffield2011conformalloopensemblesmarkovian}
Scott Sheffield and Wendelin Werner.
\newblock Conformal loop ensembles: the {M}arkovian characterization and the loop-soup construction.
\newblock {\em Ann. of Math. (2)}, 176(3):1827--1917, 2012.

\bibitem{suzuki2014convergencelooperasedrandom}
Hiroyuki Suzuki.
\newblock Convergence of loop erased random walks on a planar graph to a chordal {SLE}(2) curve.
\newblock {\em Kodai Math. J.}, 37(2):303--329, 2014.

\bibitem{uchiyama2017boundarybehaviourrwsplanar}
Kohei Uchiyama.
\newblock Boundary behaviour of {RW's} on planar graphs and convergence of {LERW} to chordal {SLE}$ _2$.
\newblock {\em arXiv preprint arXiv:1705.03224}, 2017.

\bibitem{werner2003slesboundariesclustersbrownian}
Wendelin Werner.
\newblock S{LE}s as boundaries of clusters of {B}rownian loops.
\newblock {\em C. R. Math. Acad. Sci. Paris}, 337(7):481--486, 2003.

\bibitem{10.1145/237814.237880}
David~Bruce Wilson.
\newblock Generating random spanning trees more quickly than the cover time.
\newblock In {\em Proceedings of the {T}wenty-eighth {A}nnual {ACM} {S}ymposium on the {T}heory of {C}omputing ({P}hiladelphia, {PA}, 1996)}, pages 296--303. ACM, New York, 1996.

\bibitem{Yadin_2011}
Ariel Yadin and Amir Yehudayoff.
\newblock Loop-erased random walk and {P}oisson kernel on planar graphs.
\newblock {\em Ann. Probab.}, 39(4):1243--1285, 2011.

\end{thebibliography}
\end{document}